\documentclass[a4paper,11pt,reqno]{amsart}
\usepackage{graphicx}
\usepackage[english]{babel}
\usepackage{todonotes}

\usepackage{amssymb,enumerate,amsmath}
\usepackage[bookmarks,colorlinks]{hyperref}
\usepackage{tikz}
\usepackage{array}
\usepackage{mathrsfs}
\usepackage[latin1]{inputenc}
\usepackage[T1]{fontenc}
\usepackage{ae,aecompl}
\usepackage{xypic}
\usepackage[margin=2.5cm]{geometry}

\newcommand{\ZZ}{\mathbb{Z}}
\newcommand{\NN}{\mathbb{N}}
\newcommand{\QQ}{\mathbb{Q}}

\DeclareMathOperator{\sat}{sat}
\DeclareMathOperator{\reg}{reg}
\DeclareMathOperator{\pdim}{pdim}
\DeclareMathOperator{\rk}{rk}
\DeclareMathOperator{\supp}{Supp}

\DeclareMathOperator{\rank}{rk}
\DeclareMathOperator{\Proj}{Proj}
\DeclareMathOperator{\Spec}{Spec}
\DeclareMathOperator{\Syz}{Syz}

\newcommand{\Ht}{\textnormal{Ht}}

\newcommand{\PGL}{\textsc{pgl}}

\newcommand{\Kext}{\mathbb K}
\newcommand{\Kalg}{{\Bbbk\textnormal{-Alg}}}
\newcommand{\Sets}{{\textnormal{Sets}}}
\newcommand{\schm}{{\textnormal{Sch}/\Bbbk}}

\newcommand{\cN}{\mathcal N}
\renewcommand{\geq}{\geqslant}
\renewcommand{\leq}{\leqslant}
\newcommand{\Sk}{A[\mathbf x]}
\newcommand{\Kx}{\Bbbk[\mathbf x]}
\newcommand{\kappax}{\Bbbk[\mathbf x]}
\newcommand{\T}{\mathbb{T}}

\newcommand{\redg}{\xrightarrow{G^{(s)}}}
\newcommand{\fullredg}{\xrightarrow{G^{(s)}}_*}
\newcommand{\PP}{\mathbb{P}}
\newcommand{\X}{\mathcal{X}}
\newcommand{\mult}[1]{\X_{\mathcal{P}}(#1)}
\newcommand{\nmult}[1]{\overline{\X}_{\mathcal{P}}(#1)}

\newcommand{\GrassFunctor}[1]{{\mathcal{G}\mathrm{r}}_{p(#1)}^{N_m(#1)}}
\newcommand{\GrassScheme}[1]{\mathrm{Gr}_{p(#1)}^{N_m(#1)}}
\newcommand{\GrassFunctorL}[2]{{\mathcal{G}}_{#1}^{[#2]}}

\newcommand{\QuotFunctor}{{\mathcal{Q}\mathrm{uot}}_{p(z)}^{m}}
\newcommand{\QuotFunctorReg}[1]{{\mathcal{Q}\mathrm{uot}}_{p(z)}^{m,[#1]}}
\newcommand{\QRegA}[2]{{\mathcal{Q}}_{#1}^{[#2]}}
\newcommand{\QuotScheme}{\mathrm{Quot}_{p(z)}^{m}}

\newcommand{\UC}{\mathcal{U}_{\mathcal I}^c}

\newcommand{\MFFunctor}[1]{\underline{\mathbf{Mf}}_{#1}}
\newcommand{\MFScheme}[1]{\mathbf{Mf}_{#1}}

\numberwithin{equation}{section}
\newtheorem{theorem}{Theorem}[section]
\newtheorem{corollary}[theorem]{Corollary}
\newtheorem{proposition}[theorem]{Proposition}
\newtheorem{lemma}[theorem]{Lemma}

\theoremstyle{definition}
\newtheorem{definition}[theorem]{Definition}
\newtheorem{example}[theorem]{Example}
\newtheorem{remark}[theorem]{Remark}

\begin{document}

\title{Computing Quot schemes via marked bases over quasi-stable modules}

\author[M.~Albert]{Mario Albert}

\author[C.~Bertone]{Cristina Bertone}

\author[M.~Roggero]{Margherita Roggero}

\author[W. M.~Seiler]{Werner M.~Seiler}

\address{Cristina Bertone \and Margherita Roggero\\ Dipartimento di Matematica, Universit\`a degli Studi di Torino \\ Via Carlo Alberto 10 \\ 10123 Torino \\ Italy}
\email{\href{mailto:cristina.bertone@unito.it}{cristina.bertone@unito.it}, \href{mailto:margherita.roggero@unito.it}{margherita.roggero@unito.it}}

\address{Mario Albert \and Werner M.~Seiler\\ Institut f\"ur Mathematik, Universit\"at Kassel\\ 34132 Kassel, Germany}
\email{\href{mailto:albert@mathematik.uni-kassel.de}{albert@mathematik.uni-kassel.de},\href{mailto:seiler@mathematik.uni-kassel.de}{seiler@mathematik.uni-kassel.de}}

\begin{abstract}

  Let $ \Bbbk$ be a field of arbitrary characteristic, $A$ a Noetherian
  $ \Bbbk$-algebra and consider the polynomial ring $\Sk=A[x_0,\dots,x_n]$.
  We consider homogeneous submodules of $\Sk^m$ having a special set of
  generators: a marked basis over a quasi-stable module. Such a marked
  basis inherits several good properties of a Gr\"obner basis, including a
  Noetherian reduction relation. The set of submodules of $\Sk^m$ having a
  marked basis over a given quasi-stable module has an affine scheme
  structure that we are able to exhibit. Furthermore, the syzygies of a
  module generated by such a marked basis are generated by a marked basis,
  too (over a suitable quasi-stable module in
  $\oplus^{m'}_{i=1} \Sk(-d_i)$).  We apply the construction of marked
  bases and related properties to the investigation of Quot functors (and
  schemes). More precisely, for a given Hilbert polynomial, we can
  explicitely construct (up to the action of a general linear group) an
  open cover of the corresponding Quot functor made up of open functors
  represented by affine schemes. This gives a new proof that the Quot
  functor is the functor of points of a scheme. We also exhibit a
  procedure to obtain the equations defining a given Quot scheme as a
  subscheme of a suitable Grassmannian.  Thanks to the good behaviour of
  marked bases with respect to Castelnuovo-Mumford regularity, we can adapt
  our methods in order to study the locus of the Quot scheme given by an
  upper bound on the regularity of its points.
\end{abstract}


\maketitle

\section*{Introduction}

Marked bases may be considered as a form of Gr\"obner bases not depending
on a term order.  Instead, one chooses for each generator some term as head
term such that the head terms generate a prescribed monomial ideal.  For a
long time, it was believed that it was not possible to use a marked basis
which is not a Gr\"obner basis with respect to some term order while
preserving the good features of Gr\"obner bases such as termination of the
standard normal form algorithm.  Indeed, in \cite{RS} it was shown that the
standard normal form algorithm always terminates, if and only if the head
terms are chosen according to a term order.  However, \cite{RS} contains no
results about other normal form algorithms and in \cite{BCLR, CR} it was
proven that the involutive normal form algorithm for the Pommaret division
will terminate whenever the head terms generate a strongly stable ideal
over a coefficient field of characteristic zero.

The present paper is concerned with generalizing and deepening the results
of \cite{BBR, BCLR, CR, LR2} in order to investigate \emph{Quot schemes}
over fields of arbitrary characteristic.  The Quot functor was introduced
by Grothendieck in \cite{GroHilbert}, where he also proved that this
functor is the functor of points of a projective scheme. A Hilbert scheme
is a special case of a Quot scheme. In the present paper, we consider the
Quot functor that associates to a $\Bbbk$-scheme $Z$ the set of quotients
of $\mathcal O^m_{\mathbb P^n_{Z}}$ with a given Hilbert polynomial and
flat over $Z$ (see Section \ref{sec:defquot}).  We will not exploit the
fact that the Quot functor is the functor of points of a scheme.  Actually,
we will provide an independent proof of the existence of the Quot scheme
(Corollary \ref{cor:QuotRepr}) only assuming that the Quot functor is a
Zariski sheaf \cite[Section 5.1.3]{nitsure}.

After setting some notations and recalling some useful notions and results
(Sections \ref{sec:not}, \ref{Sec:PB,QS,S}), the first part of the paper is
devoted to the investigation of the properties of marked sets, bases and
schemes over a quasi-stable module (Sections \ref{sec:markedModules} and
\ref {subsec:MarkedScheme}) and of the syzygies of a marked basis (Section
\ref{sec:syz}).

Let $ \Bbbk$ be a field \emph{of arbitrary characteristic} and $A$ a
Noetherian $ \Bbbk$-algebra.  For variables $\mathbf x:=\{x_0,\dots,x_n\}$
and a \emph{weight vector} $\mathbf d=(d_1,\dots,d_m)\in\ZZ^{m}$, we
consider homogeneous submodules of the graded $\Sk$-{module}
$\Sk^m_{\mathbf d}:=\oplus_{i=1}^m \Sk(-d_i)$.  We will define marked bases
over a \emph {quasi-stable} monomial module $U$, i.\,e.\ over a monomial
module possessing a Pommaret basis, for free submodules of
$\Sk^m_{\mathbf d}$ and investigate to what extent the algebraic properties
of Pommaret bases shown in \cite{Seiler2009II} carry over to marked
bases. It will turn out that marked bases provide us with simple bounds on
some homological invariants of the module they generate, such as Betti
numbers, (Castelnuovo-Mumford) regularity and projective dimension
(Corollary \ref {cor:bounds}). Furthermore, we prove that the set of
modules generated by a marked basis over a given quasi-stable modules has
an affine scheme structure and we exhibit an algorithmic procedure to
explicit compute equations defining this scheme (Theorem \ref{thm:rappr}).

In the second part of the paper, we use marked bases in the context of a
very specific application, namely for the derivation and the study of
equations for Quot schemes and of special loci on it, similarly to what is
shown in \cite{BBR,BLMR,BLR2013} for Hilbert schemes in the characteristic
zero case.

In Section \ref{sec:detchange}, we consider the usual action of
$g\in\PGL(n+1)$ over a finite subset $F$ of $\Sk^m$. We show that for every
finite set $F\subset \Sk^m$ we can algorithmically construct a
transformation $g\in \PGL(n+1)$ such that the set $F$ under the action of
$G$ becomes marked over a quasi-stable module.

In Section \ref{sec:defquot}, we recall the definition of the Quot functor
and its embedding in a suitable Grassmannian functor.  We first prove that,
up to the action of $\PGL$, a Grassmann functor has a cover made up of open
subsets depending on quasi-stable modules (Section \ref{sec:covergr}).  In
Section \ref{sec:coverQuot}, we intersect this open cover with a Quot
scheme and prove that, for a given Quot scheme, we have an open cover (up
to the action of $\mathrm{PGL}(n+1)$) whose open subsets are suitable
marked schemes over quasi-stable modules belonging to the Quot scheme
(Theorem \ref{thm:quasiStableCoverQuot}). The same holds if, instead of
considering the whole marked scheme, we are interested in the points
respecting an upper bound on the regularity: in this case, the open subsets
are marked schemes over quasi-stable modules that respect the bound on the
regularity too.

Starting from this open cover, we obtain in Section
\ref{sec:GlobalEquations} global equations for a Quot scheme (resp.\ its
locus defined by an upper bound on the regularity) as a closed (resp.\
locally closed) subscheme of a suitable projective space (Theorem
\ref{thm:QuotClosedInGrass}).

We end the paper with an explicit example (Section \ref{sec:example}).

\section{Notations and Generalities}\label{sec:not}

For every $n>0$, we consider the variables $x_0,\dots,x_n$, ordered as
$x_0<\cdots < x_{n-1} < x_n$ (see \cite{Seiler2009I,Seiler2009II}). This is
a non-standard way to sort the variables, but it is suitable for our
purposes. In some of the papers we refer to, variables are ordered in the
opposite way, hence the interested reader should pay attention to this when
browsing a reference.  A \emph{term} is a power product
$x^\alpha= x_0^ {\alpha_0}\cdots x_n^{\alpha_n}$. We denote by $\mathbb T$
the set of terms in the variables $x_0,\dots,x_n$.  We denote by
$\max(x^\alpha)$ the largest variable that appears with non-zero exponent
in $x^\alpha$ and, analogously, $\min(x^\alpha)$ is the smallest variable
that appears with non-zero exponent in $x^\alpha$. The \emph{degree} of a
term is $\deg(x^\alpha) =\sum_{i=0}^n \alpha_i =\vert\alpha\vert$.

Let $ \Bbbk$ be a field and $A$ be a Noetherian $ \Bbbk$-algebra. Consider the
polynomial ring $\Sk:=A[x_0,\dots,x_n]$ with the standard grading: for every
$a\in A$ we set $\deg(a)=0$. We write $\Sk_t$ for the set of
homogeneous polynomials of degree $t$ in $\Sk$. Since $\Sk=\oplus_{i\geq 0}
\Sk_i$, we define $\Sk_{\geq t}:=\oplus_{i\geq t} \Sk_i$. The ideals we
consider in $\Sk$ are always homogeneous. If $I\subset \Sk$ is a homogeneous
ideal, we write $I_t$ for $I\cap \Sk_t$ and $I_{\geq t}$ for $I\cap \Sk_{\geq
t}$. The ideal $I_{\geq t}$ is the \emph{truncation} of $I$ in degree $t$. If
$F\subset \Sk$ is a set of polynomials, we denote by $(F)$ the ideal generated
by $F$.

The ideal $J\subseteq \Sk$ is \emph{monomial} if it is generated by a set
of terms. The monomial ideal $J$ has a unique minimal set of generators
made of terms and we call it \emph{the monomial basis} of $J$, denoted by
$\mathcal B_J$. We define $\mathcal N(J)\subseteq \mathbb T$ as the set of terms in
$\mathbb T$ not belonging to $J$.  For every polynomial $f\in \Sk$,
$\supp(f)$ is the set of terms appearing in $f$ with non-zero coefficient:
$f=\sum_{x^\alpha\in \supp(f)}c_\alpha x^\alpha$, where $c_\alpha\in A$ is
non-zero. 

Hereafter, we will simply write \emph{module}
(resp. submodule) for a $\Sk$-module (resp.\ submodule of an
$\Sk$-module). For modules and submodules over other rings, we will
explicitly state the ring.

A module $M$ is \emph{graded} if $M$ has a decomposition
\begin{equation*}\label{eq:gradmod}
M=\oplus_{j\in \NN} M_j \text{ such that }\Sk_i M_j\subseteq M_{i+j}.
\end{equation*}
If $M$ is a graded module, the module $M_{\geq t}:=\oplus_{i\geq t}M_i$ is the truncation of $M$ in degree $t$.
As usual, if $M$ is a graded module, the module $M(d)$ is the graded module
(isomorphic to $M$ as a module) such that $M(d)_e=M_{d+e}$. We fix an integer $m\geq 1$ and $\mathbf
d=(d_1,\dots,d_m)\in \mathbb Z^m$. We consider the free graded $\Sk$-module
$\Sk^m_{\mathbf d}:=\oplus_{i=1}^m \Sk(-d_i)e_i$, where $e_1, \dots, e_m$ are
the standard free generators. Every  submodule of $\Sk^m_{\mathbf d}$ is
finitely generated and from now on we will only consider graded submodules of
$\Sk^m_{\mathbf d}$.

If $F$ is a set of homogeneous elements of $\Sk^m_{\mathbf d}$, we write
$\langle F\rangle$ for the graded $\Sk$-module generated by $F$ in
$\Sk^m_{\mathbf d}$. If $F=F_s$ for some positive integer $s$, we
denote by $\langle F\rangle^A$ the $A$-module generated by $F$ in
$(\Sk^m_{\mathbf d})_s$.  In particular, if $M$ is a graded submodule, every
graded component $M_j$ has the structure of  $A$-submodule in $\Sk^m_j$.

Following \cite[Chapter 15]{Eis}, a \emph{term of $\Sk^m_{\mathbf d}$} is an element
of the form $t=x^\alpha e_i$ for $i\in \{1,\dots,m\}$ and $x^\alpha\in \mathbb
T$. Furthermore, we denote by $\mathbb{T}^m$ the set of terms in
$\Sk^m_{\mathbf d}$. Observe that $\mathbb T^m=\cup_{i = 1}^m \mathbb{T}e_i$.
For $x^\alpha e_i, x^\beta e_j$ in $\mathbb T^m$ we say that $x^\alpha e_i$
divides $x^\beta e_j$ if $i=j$ and $x^\alpha$ divides $x^\beta$.

A   submodule $U$ of $\Sk^m_{\mathbf d}$ is \emph{monomial}, if it is
generated by elements in $\mathbb T^m$.  Any monomial submodule $U$ of $\Sk^m$ can
be written as
\begin{equation}\label{eq:decMonMod}
U=\oplus_{k=1}^m J^{(k)} e_k\subset \oplus \Sk(-d_k) e_k =\Sk^m_{\mathbf d},
\end{equation}
where $J^{(k)}$ is the monomial ideal generated by the terms $x^\alpha$ such
that $x^\alpha e_k\in U$. We define $\cN(U) := \cup_{k = 1}^m \cN(J^{(k)})
e_k$, where $\mathcal N(J^{(k)})\subseteq \mathbb T$.

If $M\subset \Sk^m_{\mathbf d}$ is a submodule such that for every degree $s$,
the homogeneous component $M_s$ is a free $A$-module, we define the
\emph{Hilbert function} of $M$ as $h_{M}(s)=\rk(M_s)$, which is the number of
generators contained in an $A$-basis of $M_s$. In this case, we will also say
that $M$ \emph{admits a Hilbert function}. In this setting, this definition
corresponds to the classical one (e.g. \cite[Chapter 12]{Eis}), considering
the localization of $A$ in any of its maximal ideals.  If we consider a
monomial module $U$, for every $s$, $U_s$ is always a free $A$-module and
$h_{U}(s)=\sum_{k=1}^m h_{J^{(k)}}(s)$, with $J^{(k)}$ as in
\eqref{eq:decMonMod}. 
If $M$ admits a Hilbert function, then for $s>>0$, $h_M(s)=p(s)$, where $p(z)$ is a numerical polynomial (see also \cite[Section 1]{BLMR}).

If $A= \Bbbk$, then Hilbert's Syzygy Theorem guarantees that every module
$M\subseteq \Sk^m_{\mathbf d}$ has a graded free resolution of length at most
$n$. If $A$ is an arbitrary $ \Bbbk$-algebra, there exist generally modules in
$\Sk^m_{\mathbf d}$ whose minimal free resolution has an infinite length (see
\cite[Chapter 6, Section 1, Exercise 11]{Eis}).

Assume that the module $M\subseteq \Sk^m_{\mathbf d}$ has the following graded minimal free resolution
\begin{equation}\label{eq:freeRes}
0\rightarrow E_\ell\rightarrow \cdots \rightarrow E_1\rightarrow E_0\rightarrow
M\rightarrow 0,
\end{equation}
where $E_i=\oplus_j \Sk(-j)^{b_{i,j}}$. The \emph{Betti numbers} of the
module $M$ are the set of positive integers
$\{b_{i,j}\}_{0\leq i\leq p,j\in \ZZ}$.  The module $M$ is $t$-regular if
$t\geq j-i$ for every $i,j$ such that $b_{i,j}\neq 0$. The \emph
{(Castelnuovo-Mumford) regularity} of $M$, denoted by $\reg(M)$, is the
smallest $t$ for which $M$ is $t$-regular (see for instance \cite{Eis2}).
If $M\subset \Sk^m_{\mathbf d}$ admits a Hilbert function, we recall that
$h_{\Sk^m/M}(s)=p(s)$ for all degrees $s\geq \reg(M)$.  The \emph
{projective dimension} of $M$, denoted by $\pdim(M)$, is defined as the
length of the graded minimal free resolution \eqref{eq:freeRes}, i.\,e.\ $\pdim(M) = \ell$.

From \cite[Definition 3.5.7]{KR1}, consider the ideal $\mathfrak m:=(x_0,\dots,x_n)\subset\Sk$. The \emph{saturation} of $M$, submodule of $\Sk^m$, is
\[
M^{\sat}:=M: \mathfrak m^\infty=\bigcup_{i \in \mathbb N} M: \mathfrak m^i=\{f\in \Sk^m\vert \mathfrak m^i f\subset M \text{ for some }i\in \mathbb N\}
\]


\section{Pommaret basis, Quasi-Stability and Stability}\label{Sec:PB,QS,S}

We now recall the definition and some properties of the Pommaret basis of a
monomial module.  Several of the following definitions and properties hold
in a more general setting, namely for arbitrary \emph{involutive
  divisions}. For a deeper insight into this topic, we refer to
\cite{Seiler2009I,Seiler2009II} and the references therein. 

For an arbitrary term $x^\alpha \in \mathbb{T}$, we define the following sets:
\begin{itemize}
\item the \emph{multiplicative variables of $x^\alpha$}: $\mult{x^\alpha} := \{x_i\mid x_i \leq \min(x^\alpha)\}$,
\item the \emph{non-multiplicative variables of $x^\alpha$}: $\nmult{x^\alpha} := \{x_0, \dots, x_n\} \setminus \mult{x^\alpha}$.
\end{itemize}

\begin{definition}\label{def:qsModule}
Let  $T\subset \mathbb T^m$ be a finite set of monomial generators for $U$.
  For every $\tau=x^\alpha e_k$ in $T$, we define the \emph{Pommaret cone
    in $\Sk^m_{\mathbf d}$} of $\tau$ as
  \[
    \mathcal C_{\mathcal P}^m(\tau):=\{x^\delta x^\alpha e_k\mid \delta_i=0 \, \forall x_i\in \nmult{x^\alpha} \}\subset \mathbb Te_k.
  \]
    Let $U$ be a monomial submodule of $\Sk^m_{\mathbf d}$.  We say that $T\subset \mathbb T^m$ is a \emph{Pommaret basis} of $U$ if
  \[ U\cap \mathbb T^m=\bigsqcup_{\tau \in T}\mathcal C_{\mathcal P}^m(\tau).
  \]

\end{definition}

If $U$ is a monomial module, we denote its Pommaret basis (if it exists) by
$\mathcal P(U)$.  The existence of the Pommaret basis of a monomial module
in $\Sk$ is equivalent to the concept of
\emph{quasi-stability}. In the literature, one can find a number
  of alternative names for quasi-stability (e.g. \cite{CS05,BG06,HPV03}). We recall here the definition of quasi-stable and
stable monomial modules. Both properties do not depend on the characteristic
of the underlying field.  A thorough reference on this subject is again
\cite{Seiler2009II}.

\begin{definition}\cite[Definition 4.4]{CMR15}\label{def:qstable}
Let $U\subset \Sk^m$ be a monomial module.
\begin{enumerate}[(i)]
\item\label{qstable_i} $U$ is \emph{quasi-stable} if for every term
  $x^\alpha e_k \in U\cap \mathbb T^m$ and for every non-multiplicative variable
  $x_j\in \nmult{x^\alpha}$, there is an exponent $s\geq 0$ such that
  $x_j^s x^\alpha e_k/\min(x^\alpha)\in U$.
\item $U$ is \emph{stable} if for every term $x^\alpha  e_k\in U\cap \mathbb T^m$
  and for every non-multiplicative variable $x_j\in \nmult{x^\alpha}$
  we have $x_jx^\alpha e_k/\min(x^\alpha)\in U$.
\end{enumerate}
\end{definition}

\begin{theorem}\label{thm:BorelEquiv}\cite[Proposition 4.4, Proposition 4.6]{Seiler2009II}\cite[Remark 2.10]{Mall1998}
  Let $U\subset \Sk^m$ be a monomial ideal.  $U$ is quasi-stable if and only
  if it has a (finite) Pommaret basis, denoted by $\mathcal
  P(U)$. Furthermore, $U$ is stable if and only if $\mathcal P(U)$ is its minimal monomial generating set.\\
    If $U\subset \Sk$ is quasi-stable, then $U_{\geq s}$ is  quasi-stable for every $s\geq 0$.
\end{theorem}

Recalling that any monomial module $U$ can be written as $U=\oplus_{k=1}^m J^{(k)} e_k$, with
$J^{(k)}$ suitable monomial ideals in $\Sk$ (see \eqref{eq:decMonMod}), it is immediate that $U$ is quasi-stable (resp. stable) if and only if $J^{(k)}$ is a quasi stable (resp. stable) ideal, for every $k\in \{1,\dots, m\}$.

If $U\subset \Sk^m$ is a monomial module, then a term $x^\mu e_k$ is an \emph{obstruction} to quasi-stability for $U$ if there is  $x_j >x_c=\min (x^\mu)$, such that for every $s\geq 0$, $ (x_j^s {x^\mu})e_k/ {x_c^{\mu_c}} \notin U$.
If the term $x^\mu e_k$ is an obstruction to quasi-stability for $U$, observe that in particular $\displaystyle x_j^{\mu_c}\frac{x^\mu}{x_c^{\mu_c}}e_k$ does not belong to $U$.

The following lemma collects some properties of a Pommaret basis and of the quasi-stable 
module it generates.  In particular, certain invariants of the quasi-stable module
can be directly read off from a Pommaret basis.

\begin{lemma}\label{potenze}
Let  $U$ be a quasi-stable module in $\Sk^m$.
\begin{enumerate}[(i)]
\item \label{importante_0bis} $U^{\sat}=U: (x_0)^\infty$;
\item \label{importante_0} The satiety of $U$ is the maximal degree of a
  term in $\mathcal P(U)$ which is divisible by the smallest variable in
  the polynomial ring. If $U$ is saturated, then the smallest variable of the ring does not divide any term in $\mathcal P(U)$.
\item\label{importante_00} The regularity of $U$ is the maximal degree of a
  term in $\mathcal P(U)$.
\item \label{potenze_pd} The projective dimension of $U$ is $n-D$ where $D$ is the index of the variable $\min\{\min(x^\alpha) \ | \ x^\alpha \in \mathcal{P}(J)\}$.
\item \label{potenze_ii} If $x^\eta e_k\notin U$ and $x_ix^\eta e_k \in U$, then
  either $x_ix^\eta e_k\in \mathcal P(U)$ or $x_i \in \nmult{x^\eta}$.
\item \label{potenze_iii} If $x^\eta e_k\notin U$ and
  $\left(x^\eta\cdot x^\delta\right) e_k= \left(x^{\delta'}x^\alpha\right) e_k \in U$ with
  $x^\alpha e_k\in \mathcal P(U)$ and $x^{\delta'}\in A[\mult{x^\alpha}]$,
  then $x^{\delta'} <_{lex} x^\delta$.
\end{enumerate}
\end{lemma}

\begin{proof}
  For $m=1$, items \eqref{importante_0}, \eqref{importante_00} and \eqref{potenze_pd}
  are proven in \cite[Lemma 4.11, Theorems 9.2 and 8.11]{Seiler2009II}, item
   \eqref{potenze_ii} is shown in \cite[Lemma
  3]{Bert2015}, item \eqref{potenze_iii} is a consequence of
  \eqref{potenze_ii}. We obtain the statement for $U$ applying the results for the ideals to $J^{(k)}$, $k\in \{1,\dots,m\}$ of \eqref{eq:decMonMod}.
\end{proof}

\begin{proposition}\label{prop:varieS}\
\begin{enumerate}[(i)]
\item\label{prop:varieS_iii} Let $U\subset \Sk^m$ be a quasi-stable module generated in degrees less than or equal to $s$. The module $U$ is $s$-regular if and only if $U_{\geq s}$ is stable.
\item \label{prop:varieS_iv} Let  $U$ be a quasi-stable module in $\Sk^m$  and
  consider a degree $s\geq \reg(U)$. Then $U_{\geq s}$ is stable and the set of terms $U_s\cap \T^m$ is its Pommaret basis.
\end{enumerate}
\end{proposition}
\begin{proof}
For the ideal case $m=1$ we refer to \cite[Lemma 2.2, Lemma 2.3, Theorem 9.2, Proposition 9.6]{Seiler2009II}. For the module case, we repeat the argument of Lemma \ref{prop:varieS}.
\end{proof}

\section{Marked Modules}\label{sec:markedModules}

In this section, we extend the notions of  marked polynomial, marked
basis and marked family, investigated in \cite{BCLR, CMR15, CR, LR2} for
ideals, to finitely generated modules in $\Sk^m_{\mathbf d}$.  Let
$U\subset \Sk^m_{\mathbf d}$ be a monomial module, so that
$U=\oplus_{k=1}^m J^{(k)} e_k$, with $J^{(k)}$ monomial ideal in $\Sk$.  If
$U$ is a quasi-stable module, we denote by $\mathcal{P}(U)$ the Pommaret
basis of $U$.

\begin{definition}\cite{RS}
  A \emph{marked polynomial} is a polynomial $f\in \Sk$ together with a
  fixed term $x^\alpha$ in $\supp(f)$ whose coefficient is equal to $1_A$.
  This term is called \emph{head term} of $f$ and denoted by $\Ht(f)$.
  With a marked polynomial $f$, we associate the following sets:
  \begin{itemize}
  \item the \emph{multiplicative variables of $f$}: $\mult{f} := \mult{\Ht(f)}$;
  \item the \emph{non-multiplicative variables of $f$}: $\nmult{f} := \nmult{\Ht(f)}$.
  \end{itemize}
\end{definition}

\begin{definition}
  A \emph{marked homogeneous module element} is a homogeneous module
  element in $\Sk^m_{\mathbf d}$ with a fixed term in its support whose
  coefficient is $1_A$ and which is called \emph{head term}. More
  precisely, a marked homogeneous module element is of the form
  \begin{equation*}
    f_\alpha^k = f_\alpha e_k - \sum_{l\neq k} g_l e_l \in \Sk^m_{\mathbf d}
  \end{equation*}
  where $f_\alpha$ is a marked polynomial with $\Ht(f_\alpha)=x^\alpha$,
  and $\Ht(f_\alpha^k) = \Ht(f_\alpha) e_k = x^\alpha e_k$.
\end{definition}

The following definition is fundamental for this work.  It is modelled on
a well-known characteristic property of Gr\"obner bases.

\begin{definition}\label{def:MarkedSet}
  Let $T \subset \mathbb{T}^m$ be a finite set and $U$ the module generated
  by it in $\Sk^m_{\mathbf d}$.  A \emph{$T$-marked set} is a finite set
  $G \subset \Sk^m_{\mathbf d}$ of marked homogeneous module elements
  $f_\alpha^k$ with $\Ht(f_\alpha^k) = x^\alpha e^k \in T$ and
  $\supp(f_\alpha^k - x^\alpha e_k) \subset \langle \cN (U) \rangle$
  (obviously, $|G|=|T|$).

  The $T$-marked set $G$ is a \emph{$T$-marked
    basis}, if $\mathcal{N}(U)_s$ is a basis of
  $\left(\Sk^m_{\mathbf d}\right)_s/\langle G\rangle_s$ as $A$-module, i.\,e.\ if
  $\left(\Sk^m_{\mathbf d}\right)_s = \langle G\rangle_s \oplus \langle \mathcal
  {N}(U)_s \rangle^A$ for all $s$.
\end{definition}

\begin{lemma}\label{lem:markedOnMon}
  Let $T\subset \mathbb T^m$ be a finite set and $U$ the module generated
  by it in $\Sk^m_{\mathbf d}$. Let $M\subseteq \Sk^m_{\mathbf d}$ be a
  module such that for every $s$ the set $\mathcal N(U)_s$ generates the
  $A$-module $(\Sk^m_{\mathbf d})_s/M_s$. Then for every degree $s$ there
  exists a $U_s\cap \mathbb T^m$-marked set $F=F_s$ contained in $M_s$ such
  that
  \[
    \left(\Sk^m_{\mathbf d}\right)_s=\langle F\rangle^A\oplus \langle\mathcal N(U)_s\rangle^A.
  \]
\end{lemma}

\begin{proof}
  Let $\pi$ be the usual projection morphism of $\Sk^m_{\mathbf d}$ onto
  the quotient $\Sk^m_{\mathbf d}/M$. For every
  $x^\alpha e_k\in U_s\cap \mathbb T^m$, we consider $\pi(x^\alpha e_k)$
  and choose a representation
  $\pi(x^\alpha e_k)=\sum_{x^\eta e_l\in \mathcal N(U)_s}c_{\eta l}^{\alpha
    k}x^\eta e_l$, $c_{\eta l}^{\alpha k}\in A$, which exists as
  $\mathcal N(U)_s$ generates $(\Sk^m_{\mathbf d})_s/M_s$ as an
  $A$-module. We consider the set of marked module elements
  $F=\{f_\alpha^k\}_{x^\alpha e_k\in U_s}$, where
  $f_\alpha^k:=x^\alpha e_k-\pi(x^\alpha e_k)$ and
  $\Ht(f_\alpha^k)=x^\alpha e_k$.

  We now prove that
  $\Sk^m_s=\langle F\rangle^A\oplus \langle\mathcal N(U)_s\rangle^A$. We
  first prove that every term in $\mathbb T^m_s$ belongs to
  $\langle F\rangle^A+ \langle\mathcal N(U)_s\rangle^A$. If
  $x^\beta e_l \in \mathcal N(U)_s$, there is nothing to prove. If
  $x^\beta e_l\in U_s$, then there is $f_\beta^l\in F$ such that
  $\Ht(f_\beta^l)=x^\beta e_l$, hence we can write
  $x^\beta e_l=f_\beta^l+(x^\beta e_l -f_\beta^l)=f_\beta^l+\pi(x^\beta
  e_l)$.

  We conclude proving that
  $\langle F\rangle^A\cap \langle\mathcal N(U)_s\rangle^A=\{0^m_A\}$. Let
  $g\in \Sk^m_{\mathbf d}$ be an element belonging to
  $\langle F\rangle^A\cap \langle\mathcal N(U)_s\rangle^A$:
  $g=\sum_{f_\alpha^k\in F} \lambda_{\alpha k} f_\alpha^k\in \langle
  \mathcal N(U)_s \rangle$. Since the head terms of $f_\alpha^k$ cannot
  cancel each other, $\lambda_{\alpha k}=0$ for every $\alpha$ and $k$ and
  hence $g=0$.
\end{proof}

We specialize now to the case that $U$ is a quasi-stable module and
$T=\mathcal{P}(U)$ its Pommaret basis.  We study a reduction relation
naturally induced by any basis marked over such a set $T$.  In particular,
we show that it is confluent and Noetherian just as the familiar reduction
relation induced by a Gr\"obner basis.

\begin{definition}
Let $U\subseteq \Sk^m_{\mathbf d}$ be a quasi-stable module and $G$ be a $\mathcal{P}(U)$-marked set
in $\Sk^m_{\mathbf d}$. We introduce the following sets:
\begin{itemize}
\item $G^{(s)} := \left\{ x^\delta f_\alpha^k \ \big\vert\ f_\alpha^k \in G, x^\delta \in A[\mult{f_\alpha^{k}}], \deg x^\delta f_\alpha^k=s \right\};$
\item $\widehat{G}^{(s)}:=\left\{ x^\delta f_\alpha^k\ \big\vert\  f_\alpha^k \in G, x^\delta\notin  A[\mult{f_\alpha^{k}}], \deg x^\delta f_\alpha^k=s  \right\}=\left\{ x^\delta f_\alpha^k\ \big\vert\  f_\alpha^k \in G\right\}\setminus G^{(s)} $;
\item $\cN (U, \langle G\rangle):= \langle G\rangle \cap \langle \cN(U)\rangle$.
\end{itemize}

\end{definition}

\begin{lemma}\label{lem:modulelexorder}
Let $U\subseteq\Sk^m_{\mathbf d} $ be a quasi-stable module and $G$ a $\mathcal{P}(U)$-marked
set. For every product
$x^\delta f_\alpha^k$ with $f_\alpha^k\in G$,
each term in $\supp(x^\delta x^\alpha e_k - x^\delta f_\alpha^k)$ either
belongs to $\mathcal{N}(U)$ or is of the form $x^\eta  x^\nu e_l \in \mathcal C_{\mathcal P}^m(x^\nu e_l)$ with $x^\nu e_l \in \mathcal P(U)$ and   $x^\eta <_{lex} x^\delta$.
\end{lemma}
\begin{proof}
  It is sufficient to consider
  $x^\delta x^\beta e_l\in \supp(x^\delta x^\alpha e_k-x^\delta
  f_\alpha^k)\cap U$. Then $x^\delta x^\beta \in J^{(l)}$ for some
  quasi-stable ideal $J^{(l)}\subset \Sk$ appearing in
  \eqref{eq:decMonMod}.  Therefore there exists
  $x^\gamma \in \mathcal{P}(J^{(l)})$ such that
  $x^\delta x^\beta\in \mathcal C_{\mathcal P}(x^\gamma)$. More precisely,
  if $x^\eta:=x^\delta x^\beta/x^\gamma$, then $x^\eta<_{lex} x^\delta$ by
  Lemma \ref{potenze} \eqref{potenze_iii}.
\end{proof}

Note in the next definition the use of the set $G^{(s)}$, which means that
we use here a gene\-ra\-li\-zation of the involutive reduction relation
associated with the Pommaret division and not of the standard reduction
relation in the theory of Gr\"obner bases.  This modification is the key
for circumventing the restrictions imposed by the results of \cite{RS}.  It
also entails that if a term is reducible, then there is only one element in
the marked basis which can be used for its reduction.

\begin{definition}\label{def:reduction}
Let $U\subseteq \Sk^m_{\mathbf d}$ be a quasi-stable module and $G$ a $\mathcal{P}(U)$-marked set. We
 denote by $\redg$ the transitive closure of the relation $h \redg h - \lambda x^\eta
f_\alpha^k$ where $x^\eta x^\alpha e_k$ is a term that appears in $h$ with
a non-zero coefficient $\lambda\in A$, which satisfies $\deg(x^\eta x^\alpha e_k)=s$ and $x^\eta f_\alpha^k\in G^{(s)}$. 
We will write $h \fullredg g$ if $h \redg g$ and $g\in \langle \mathcal{N}(U)\rangle$. Observe that if $h\in \left(\Sk^m_{\mathbf d}\right)_s$, then $h\redg g\in \left(\Sk^m_{\mathbf d}\right)_s$.
\end{definition}

\begin{proposition} \label{prop:ConfNoeth}Let $U\subseteq \Sk^m_{\mathbf d}$ be a quasi-stable module and $G$ a $\mathcal{P}(U)$-marked set.
The reduction relation $\redg$ is Noetherian.
\end{proposition}
\begin{proof}
It is sufficient to prove that for every term $x^\gamma e_k$ in $U$, there is $g\in \langle \mathcal{N}(U)\rangle^{A}$ such that $x^\gamma e_k \fullredg g$.

Since $x^\gamma e_k\in U$, there exists a unique $x^\delta f_\alpha^k\in
G^{(s)}$ such that $x^\delta \Ht(f_\alpha^k)=x^\gamma e_k$. Hence,
$x^\gamma e_k\redg x^\gamma e_k -x^\delta f_\alpha^k$. If we could proceed
in the reduction without ever obtaining an element in $\langle \mathcal
N(U)\rangle$, we would obtain by Lemma \ref{lem:modulelexorder} an infinite
lex-descending chain of terms in $\mathbb T$ which is impossible since lex
is a well-ordering. Hence $\redg$ is Noetherian. 
\end{proof}

\begin{corollary}\label{lem:markedlexorder}
Let $U\subseteq \Sk^m$ be a quasi-stable module and $G$ be a
$\mathcal{P}(U)$-marked set.  Every term $x^\beta e_k\in \mathbb T^m_s$ of
degree $s$ can be expressed in the form
\begin{equation}\label{eq:forma}
x^\beta e_l=\sum \lambda x^{\delta}f_{\alpha}^{k}+g,
\end{equation}
where $\lambda\in A\setminus\{0_A\}$,
$x^{\delta}f_{\alpha}^{k}\in G^{(s)}$, $g\in \langle
\mathcal{N}(U)\rangle^{A}$ and  the terms $x^{\delta}$ form a
sequence which is strictly descending with respect to lex.
\end{corollary}
\begin{proof}
For terms in $\mathcal N(U)$, there is nothing to prove. For $x^\beta e_l\in U$, it is sufficient to consider $g\in \langle \mathcal N(U)\rangle^{A}$ such that $x^\beta e_l\fullredg g$. The polynomials $x^{\delta}f_{\alpha}^{k}\in G^{(s)}$ are exactly those used during the reduction $\redg$. They fulfill the statement on the terms $x^{\delta}$ by Lemma \ref{lem:modulelexorder}.
\end{proof}

We now put an order on the polynomials $x^{\delta}f_{\alpha}^{k}\in G^{(s)}$, assuming that the polynomials in $G$ are ordered (in some way):
$x^\delta f_\alpha^k \prec x^{\delta'}f_{\alpha'}^{k'}$ if $f_{\alpha}$ is smaller then $f_{\alpha'}$ or if $f_{\alpha}^k=f_{\alpha'}^{k'}$ and $x^{\delta}<_{\mathrm lex}x^{\delta'}$.\\
In the following, when we say that a polynomial in a subset of $G^{(s)}$ is maximal, we refer to this order.

\begin{lemma}\label{lem:notzero}
Let $U\subseteq \Sk^m_{\mathbf d}$ be a quasi-stable module and $G$ be a $\mathcal{P}(U)$-marked set. 
Consider a homogeneous element $g\in \Sk^m_{\mathbf d}$ such that $h = \sum \lambda  x^{\delta}f_{\alpha}^{k}$, with $\lambda \in A
\setminus \{0\} $ and $x^{\delta}f_{\alpha}^{k}\in G^{(s)}$ with $s=\deg(h)$ and $x^{\delta}f_{\alpha}^{k}$ pairwise different. Then $h\neq  0_A^m$ and $h\notin \langle \mathcal N(U)\rangle^A$.
\end{lemma}
\begin{proof}

Let $x^{\overline \delta}f_{\overline \alpha}^{\overline k}$  be the maximal polynomial of $G^{(s)}$ appearing in the summation $\sum \lambda  x^{\delta}f_{\alpha}^{k}=0$ with $\lambda\neq 0$.

Then, by Lemma \ref{lem:modulelexorder}, the term $x^{\overline \delta}x^{\overline \alpha} e_{\overline k}$, does not appear in the support of the other polynomials $x^{\delta} f_{\alpha}^k$ involved in the summation. Hence, $h\neq 0_A^m$. Furthermore,  $x^{\overline \delta}x^{\overline \alpha} e_{\overline k}\in U$ belongs to the support of $h$, hence $h$ does not belong to $\langle \mathcal N(U)\rangle$.
\end{proof}

\begin{proposition} \label{prop:Conf}Let $U\subseteq \Sk^m_{\mathbf d}$ be a quasi-stable module and $G$ a $\mathcal{P}(U)$-marked set.
The reduction relation $\redg$ is confluent.
\end{proposition}
\begin{proof}
Let $h$ be a polynomial in $\Sk^m$ and we reduce it twice with $\redg$, following different paths along the reduction: $h\fullredg g_1\in\langle \mathcal N(U)\rangle$ and $h\fullredg g_2\in\langle \mathcal N(U)\rangle$.
By Corollary \ref{lem:markedlexorder} applied on the terms of the suport of $h$, 
\begin{equation}\label{eq:confluency}
h= \sum \lambda x^{\delta}f_{\alpha}^{k}+g_1=\sum \mu x^{\delta}f_{\alpha}^{k}+g_2. 
\end{equation}

Then $g_1-g_2=\sum (\lambda-\mu)x^{\delta}f_{\alpha}^{k}$. If there is $\lambda-\mu\in A\setminus\{0\}$, then by Lemma \ref{lem:notzero}, $g_1-g_2\notin \langle\mathcal N(U)\rangle$, against the hypothesis. Then $\lambda=\mu$ for every $x^\delta f_\alpha^k \in G^{(s)}$ in \eqref{eq:confluency} and $g_1=g_2$.
\end{proof}

\begin{corollary}\label{cor:unicity}
Let $U\subseteq \Sk^m$ be a quasi-stable module and $G$ be a
$\mathcal{P}(U)$-marked set.  Every term $x^\beta e_k\in \mathbb T^m_s$ of
degree $s$ has a unique form of the type in \eqref{eq:forma}.
\end{corollary}

The following Theorem and Corollaries collect some basic properties of sets
marked over a Pommaret basis.  They generalize analogous statements in
\cite[Theorems 1.7, 1.10]{LR2} which consider only ideals and marked bases
where the head terms generate a strongly stable ideal.

\begin{theorem}\label{th:markedSetChar}
Let $U\subset \Sk^m_{\mathbf d}$ be a quasi-stable module, with $q(s):=\rk(U_{s})$, and  $G$ a
$\mathcal{P}(U)$-marked set.  Then, we have for every degree $s$ the
following decompositions of $A$-modules:
\begin{enumerate}[(i)]
\item\label{it:markedSetChar_i} $\langle G\rangle_s = \langle G^{(s)} \rangle^A + \langle \widehat{G}^{(s)} \rangle^A $;
\item\label{it:markedSetChar_ii} $(A[\mathbf x]^m_{\mathbf d})_s =\langle G^{(s)}\rangle^A\oplus\langle\cN(U)_s\rangle^A$;
\item\label{it:markedSetChar_iii} the $A$-module $\big\langle G^{(s)} \big\rangle^A$ is free of rank equal to $\vert G^{(s)} \vert =\rank(U_s)$ and it is generated (as an $A$-module) by a  unique $U_s\cap \mathbb T^m$-marked set $\widetilde{G}^{(s)}$;
\item\label{it:markedSetChar_iv} $\langle G\rangle_s= \big\langle {G}^{(s)} \big\rangle^A \oplus  \cN (U, \langle G \rangle)_s  $.
\end{enumerate}
Moreover,   the following conditions are equivalent:
\begin{enumerate}[(i)]\setcounter{enumi}{4}
\item\label{it:markedSetChar_v} $G$ is a $\mathcal P(U)$-marked basis;
  \item\label{it:markedSetChar_vi} for all degrees $s$, $\langle G\rangle_s=\big\langle {G}^{(s)}\big\rangle^A $;
\item\label{it:markedSetChar_vii}  $\cN(U,\langle G\rangle)=\{0^m_A\}$;
 \item\label{it:markedSetChar_ix} for all $s$, $\bigwedge^{q(s)+1} \langle G\rangle_{s}=0_A$.
\end{enumerate}
\end{theorem}

\begin{proof}
Item \eqref{it:markedSetChar_i}: immediate.

Item \eqref{it:markedSetChar_ii} is a consequence of Corollary \ref{cor:unicity}.

Item \eqref{it:markedSetChar_iii}:
we use the arguments of \cite[Theorem 1.7]{LR2} for the ideal
case: by \eqref{it:markedSetChar_ii} we have the short exact sequence
\begin{equation*}
0 \longrightarrow \langle G^{(s)} \rangle \hookrightarrow (A[\mathbf x]^m_
{\mathbf d})_s
\xrightarrow{\; \pi \;} \langle \cN(U)_s \rangle \longrightarrow 0\;.
\end{equation*}
For each $x^\alpha e_k$ in $U_s$ we compute the image $\pi(x^\alpha e_k) =
\sum_{x^\beta e_l \in \cN(U)_s} a_{\alpha\beta k l} x^\beta e_l$ and consider
the set $\tilde{G}^{(s)} := \{\tilde{f}_\alpha^k := x^\alpha - \sum_{x^\beta e_l
\in \cN(U)_s} a_{\alpha\beta k l} x^\beta e_l \mid x^\alpha e_k \in U_s\} \subseteq \ker \pi =
\langle G^{(s)} \rangle$. Let $U^\prime := \langle U_s \rangle$. By construction,
$\tilde{G}^{(s)}$ is a $U^\prime$-marked set with $\Ht(\tilde{f}_\alpha^k) =
x^\alpha e_k$. Applying \eqref{it:markedSetChar_ii} to this $U^\prime$-marked
set, we have $\langle \tilde{G}^{(s)} \rangle + \langle \cN(U^\prime)_s \rangle
= (A[\mathbf x]^m_{\mathbf d})_s$.

Finally, since the $A$-module generated by $\tilde{G}^{(s)}$ is contained in
$\langle G^{(s)} \rangle$ and $\cN(U)_s = \cN(U^\prime)_s$, the modules
$\langle \tilde{G}^{(s)} \rangle$ and $\langle G^{(s)} \rangle$ coincide. Note
that $\tilde{G}^{(s)}$  is marked on the monomial module $U^\prime$ which is
generated by $U_s$, but does not need to be a $U_{\geq s}$-marked set, since
$U_ {\geq s}$ may have minimal generators of degree greater than $s$.

Item \eqref{it:markedSetChar_iv}: by items \eqref{it:markedSetChar_i} and
\eqref{it:markedSetChar_iii}, we have $\langle G\rangle_s=\langle
\widetilde{G}^{(s)}\rangle^A +\langle
\widehat{G}^{(s)}\rangle^A$. Recalling that $\langle
\widetilde{G}^{(s)}\rangle^A\cap \langle \mathcal
N(U)_s\rangle^A=\{0_A^m\}$ by Lemma \ref{lem:markedOnMon}, it is sufficient
to show that every $g\in \langle \widehat{G}^{(s)}\rangle^A$ can be
written $g=f+h$ with $f\in \langle \widetilde{G}^{(s)}\rangle^A$ and $h \in
\langle \mathcal N(U)_s\rangle^A$: we express every term $x^\beta e_l\in
U_s$ appearing in $g$ with non-zero coefficient in the form $x^\beta e_l=\widetilde{f}^l_\beta+(x^\beta e_l-\widetilde{f}^l_\beta)$ where $\widetilde{f}^l_\beta$ is the unique polynomial in $\widetilde{G}^{(s)}$ with $\Ht(\widetilde{f}^l_\beta)=x^\beta e_l$. By construction, $h\in \mathcal N(U,\langle G\rangle)_s$. By item \eqref{it:markedSetChar_ii}, we obtain the assertion.

Items \eqref{it:markedSetChar_v}, \eqref{it:markedSetChar_vi},
\eqref{it:markedSetChar_vii} are equivalent by the previous items. In fact these
properties are a rephrasing of the definition of $\mathcal{P}(U)$-marked basis.

With respect to \cite{LR2}, the only new item is \eqref{it:markedSetChar_ix}, which  is obviously equivalent to \eqref{it:markedSetChar_vi} and
\eqref{it:markedSetChar_vii}.
In fact,  by   \eqref{it:markedSetChar_iii} and
\eqref{it:markedSetChar_iv} we find that $\langle G\rangle_{s}=\big\langle {G}^{(s)}\big\rangle^A \oplus
\cN (U, \langle G\rangle)_{s} $
and $\rank \big\langle {G}^{(s)}\big\rangle^A =\rk (U_{s})=q(s)$.
\end{proof}

\begin{remark}\label{rem:HFunc}
 If $G\subset \Sk^m$ is a $\mathcal P(U)$-marked basis, then, by Theorem \ref{th:markedSetChar} \eqref{it:markedSetChar_ii}, \eqref{it:markedSetChar_iii}  and \eqref{it:markedSetChar_vi},   the $\Sk$-module $\langle G\rangle$ admits a Hilbert function, which is the same as  the Hilbert
  function of the monomial module $U$.
\end{remark}

\begin{corollary}\label{cor:finiteCond} Let $U\subset \Sk^m_{\mathbf d}$ be a quasi-stable module and  $G$ be a $\mathcal{P}(U)$-marked set. 
The following conditions are equivalent:
\begin{enumerate}[(i)]
\item \label{finiteCond_i}$G$ is a $\mathcal P(U)$-marked basis;
\item \label{finiteCond_ii}$\langle G\rangle_s=\big\langle {G}^{(s)}\big\rangle^A $ for every $s\leq \reg(U) + 1$;
\item \label{finiteCond_iii}$\cN(U,\langle G\rangle)_s=\{0^m_A\}$ for every $s\leq \reg(U) + 1$;
\item \label{finiteCond_iv}$\bigwedge^{q(s)+1} \langle G\rangle_{s}=0_A$ for every $s\leq \reg(U) + 1$ .
\end{enumerate}
\end{corollary}
\begin{proof}
By the second part of Theorem \ref{th:markedSetChar}, item
\eqref{finiteCond_i} implies item \eqref{finiteCond_ii} and items
\eqref{finiteCond_ii}, \eqref {finiteCond_iii} and \eqref{finiteCond_iv} are
equivalent.

For the proof that item \eqref{finiteCond_ii} implies
\eqref{finiteCond_i}, we follow the arguments used in \cite[Theorem 1.10]{LR2} and
adapt them to the module case. We have to prove that $(A[\mathbf x]^m_
{\mathbf d})_s = \langle G \rangle_s \oplus \langle \cN(U)_s \rangle$ for
every $s$. This is true for $s \leq m+1$ by hypothesis. By Theorem
\ref{th:markedSetChar} \eqref {it:markedSetChar_ii},
\eqref{it:markedSetChar_iii}, we know that $(A[\mathbf x]^m_{\mathbf d})_s =
\langle G^{(s)} \rangle \oplus \langle \cN(U)_s \rangle$ and $\langle G^{(s)}
\rangle \subseteq \langle G \rangle_s$, so that we have to prove $\langle G
\rangle_s \subseteq \langle G^{(s)} \rangle$. Let us assume that this is not
true and let $t$ be the minimal degree for which $\langle G \rangle_t
\not\subseteq \langle G^{(t)} \rangle$. Note that $t \geq m+2 > m$ and
$\langle G \rangle_t = x_0 \langle G \rangle_{t-1} + \dots + x_n \langle G
\rangle_{t-1}$. Since $\langle G \rangle_{t-1} = \langle G^{(t-1)} \rangle$,
there must exist a variable $x_i$ such that $x_i \langle G \rangle_{t-1}
\not\subseteq \langle G^{(t)} \rangle$ or equivalently $x_i \langle G^{(t-1)}
\rangle \not\subseteq \langle G^{(t)} \rangle$. Assume that the index $i$ is
minimal with this property and take a polynomial $x^\delta f_\alpha^k \in G^{(t-1)}$ with
$x^\alpha e_k = \Ht(f_\alpha^k) \in \mathcal{P}(U)$ such that $x_i x^\delta
f_\alpha^k \notin \langle G^{(t)} \rangle$. The variable $x_i$ has to be
greater than $\min(x^\alpha)$, since otherwise $x_i x^\delta f_\alpha^k \in
G^{(t)}$. Morevover $|\delta| > 0$ since $t-1 > m$. Let $x_j = \max(x^\delta)
\leq \min (x^\alpha) < x_i$ and $x^{\delta^\prime} = \frac{x^\delta}{x_i}$.
The polynomial is contained in $\langle G \rangle_{t-1}$ while $x_j(x_i
x^{\delta^\prime} f_\alpha^k) = x_i x^\delta f_\alpha^k$ is not contained in
$\langle G^{(t-1)} \rangle$, contradicting the minimality of $i$.
\end{proof}

\begin{corollary}
 Let $U\subset \Sk^m_{\mathbf d}$ be a quasi-stable module, such that $U=\oplus J^{(k)} e_k$ with $J^{(k)}$ saturated ideal for every $k$,  and  $G$ be a $\mathcal{P}(U)$-marked set. Then the following conditions are equivalent:
\begin{enumerate}[(i)]
\item\label{cor:markedSetSat_0}$G$ is a $\mathcal P(U)$-marked basis;
\item\label{cor:markedSetSat_i} $\langle G\rangle_{\reg(U)+1}=\big\langle {G}^{(\reg(U)+1)}\big\rangle^A $;
\item\label{cor:markedSetSat_ii}  $\cN(U,\langle G\rangle)_{\reg(U)+1}=\{0^m_A\}$;
\item\label{cor:markedSetSat_iii} $\bigwedge^{Q+1} \langle G\rangle_{\reg(U)+1}=0_A$, where $Q:=\rk(U_{\reg(U)+1})$.
\end{enumerate}
\end{corollary}
\begin{proof}
The equivalence among items \eqref{cor:markedSetSat_i}, \eqref{cor:markedSetSat_ii} and \eqref{cor:markedSetSat_iii} is immediate by Theorem \ref{th:markedSetChar}. We only prove that items \eqref{cor:markedSetSat_0} and \eqref{cor:markedSetSat_ii} are equivalent. If $G$ is a $\mathcal P(U)$-marked basis, then   by Theorem \ref{th:markedSetChar} we have $\cN(U,\langle G\rangle)_{\reg(U)+1}=\{0^m_A\}$.

We now assume that $\cN(U,\langle G\rangle)_{\reg(U)+1}=\{0^m_A\}$ and prove that $\cN(U,\langle G\rangle)=\{0^m_A\}$ . By Corollary \ref{cor:finiteCond}, it is sufficient to prove that  $\cN(U,\langle G\rangle)_{s}=\{0^m_A\}$  for every $s\leq \reg(U)$. If $f\in \cN(U,\langle G\rangle)_{s}$, with $s\leq \reg(J)$, then $x_0^{\reg(U)+1-s}f\in \cN(U,\langle G\rangle)_{\reg(U)+1}$,    by  Lemma \ref{potenze} \eqref{importante_0} and \eqref{potenze_ii} applied to $U$. Hence $f=0_A^m$.
\end{proof}

\begin{corollary}\label{LemmaCri}
Let $U\subset \Sk^m_{\mathbf d}$ be a quasi-stable module and $W\subset \Sk^m_{\mathbf d}$ be a finitely generated graded submodule such that $(\Sk^m_{\mathbf d})_s= W_s\oplus \langle \mathcal N(U)_s\rangle^A$ for every $s$. Then $W$ is generated by a $\mathcal P(U)$-marked basis.
\end{corollary}

\begin{proof} The statement is an easy consequence   of Theorem  \ref{th:markedSetChar} as soon as we define a $\mathcal P(U)$-marked set  generating $W$.

 By the hypotheses, for every degree $s$ and every monomial $x^\alpha e_k \in
 \mathcal P(U)$  there is a unique element $h_\alpha^k \in \langle \mathcal N
 (U)_s\rangle^A $ such that  $  x^\alpha e_k -h_\alpha^k \in W_s$.

 The collection  $G$   of  the elements $  x^\alpha e_k -h_\alpha^k$ is obviously a $\mathcal P(U)$-marked set and generates a graded  submodule of $W$. Moreover,  $(\Sk^m_{\mathbf d})_s=W_s\oplus \langle \mathcal N(U)_s\rangle^A= \langle G^{(s)}\rangle^A\oplus \langle \mathcal N(U)_s\rangle^A$.  Therefore,  $W_s=\langle G^{(s)}\rangle^A\subseteq  G_s\subseteq  W_s$, so that $G$ generates $W$ as a graded $\Sk$-module.
\end{proof}

Finally, we give an algorithmic  method to check whether a marked set is a marked
basis using the reduction process introduced in Definition \ref{def:reduction}.

\begin{theorem}\label{thm:algoMBasis}
Let $U\subset \Sk^m_{\mathbf d}$ be a quasi-stable module and  $G$ be a $\mathcal{P}(U)$-marked set. The set $G$ is a $\mathcal P(U)$-marked basis if and only if
\[ \forall f_\alpha^k \in  G,\forall x_i\in
\nmult{f_\alpha^k} : x_if_\alpha^k\xrightarrow{\ G^{(s)}} 0^m_A.
\]
\end{theorem}
\begin{proof}
We adapt the arguments used in \cite[Theorem 5.13]{CMR15} for the ideal case.
Since ``$\Rightarrow$'' is a consequence of Theorem \ref{th:markedSetChar}, we
only prove ``$\Leftarrow$''. More precisely, we prove that $\langle G \rangle_s
= \langle G^{(s)} \rangle$ showing that if $f_\alpha^k \in G$ and
$\deg(x^{\alpha + \delta}) = s$, then $x^\delta f_\alpha^k$ is either an
element of $G^{(s)}$ itself or a linear combination of polynomials in
$G^{(s)}$.

If this were not true, we can choose an element $x^\delta f_\alpha^k \in
\langle G^{(s)} \rangle$ with $x^\delta$ minimal with respect to $<_{lex}$. As
$x^\delta f_\alpha^k \notin G^{(s)}$, at least one variable $x_i$ appearing in
$x^\delta$ with a non-zero exponent is non-multiplicative for $x^\alpha$. Let
$x^\delta = x_i x^{\delta^\prime}$. By hypothesis, $x_i f_\alpha^k \fullredg
0$, so that $x_i f_\alpha^k$ is a linear combination $\sum c_i x^{\eta_i} f_
{\beta_i}^{k_i}$ of polynomials in $G^{(|\alpha| + 1)}$. By Lemma \ref
{lem:modulelexorder}, we have $x^{\eta_i} <_{lex} x_i$.

Now $x^\delta f_\alpha^k = x^{\delta^\prime}(x_i f_\alpha^k) = x^
{\delta^\prime} (\sum c_i x^{\eta_i} f_{\beta_i}^{k_i}) = \sum c_i x^{\eta_i +
\delta^\prime} f_{\beta_i}^{k_i}$, where $x^{\eta_i \delta^\prime} <_ {lex}
x_i x^{\delta^\prime} = x^\delta$. This yields a contradiction, since
$x^{\eta_i + \delta} f_{\beta_i}^{k_i} \in \langle G^{(s)} \rangle$ by the
minimality of $x^\delta$.
\end{proof}

\section{The scheme structure of $\mathrm{Mf}(\mathcal P(U))$}\label{subsec:MarkedScheme}

We now exhibit a natural scheme structure on the set containing all
modules generated by a $\mathcal P(U)$-marked basis with $U$ a quasi-stable
module as in the previous section. If $\sigma:
A\rightarrow B$ is a morphism of $\Bbbk$-algebras, we will also call $\sigma$ its natural extension  to a morphism $A[\mathbf x]\rightarrow B[\mathbf x]$.

 We consider the functor
of the marked bases on $\mathcal P(U)$ from the category of Noetherian
$ \Bbbk$-algebras to the category of sets
\[
\MFFunctor{\mathcal P(U)}^{\, m,  \mathbf d}: \underline{\text{Noeth}\  \Bbbk\!\!-\!\!\text{Alg}} \longrightarrow \underline{\text{Sets}}
\]
that associates to any Noetherian $ \Bbbk$-algebra  $A$ the set
$$
\MFFunctor{\mathcal P(U)}^{\, m,  \mathbf d}(A):=\{ G\subset \Sk^m_{\mathbf d}\mid G\text{ is a } \mathcal P(U)\text{-marked basis}\}\,,$$
or, equivalently by Corollary \ref{LemmaCri},
$$\MFFunctor{\mathcal P(U)}^{\, m,  \mathbf d}(A):=\{W\subset \Sk^m_{\mathbf d}\mid W \text{ is generated by a $\mathcal P(U)$-marked basis}\}\,,$$
and to any morphism $\sigma: A \rightarrow B$ the map
\[
\renewcommand{\arraystretch}{1.3}
\begin{array}{rccc}
\MFFunctor{\mathcal P(U)}^{\, m,  \mathbf d}(\sigma):& \MFFunctor{\mathcal P(U)}^{\, m,  \mathbf d}(A) &\longrightarrow& \MFFunctor{\mathcal P(U)}^{\, m,  \mathbf d}(B)\\
&G & \longmapsto& \sigma(G)\;.
\end{array}
\]
Note that the image $\sigma(G)$ under this map is indeed again a
$\mathcal P(U)$-marked basis, as we are applying the functor $-\otimes_A B$
to the decomposition
$(A[\mathbf x]^m_{\mathbf d})_s =\langle
G^{(s)}\rangle^A\oplus\langle\cN(U)_s\rangle^A$ for every degree $s$.

The above introduced functor turns out to be representable by an affine
scheme that can be explicitly constructed by the following procedure.\\
  We
consider the $ \Bbbk$-algebra $ \Bbbk[C]$ where $C$ denotes the finite set of
variables
$\bigl\{C_{\alpha\eta k l}\mid x^\alpha e_k \in \mathcal P(U ), x^\eta
e_l\in \mathcal N(U ), \deg(x^\eta e_l)=\deg(x^\alpha e_k)\bigr\}$ and
construct the $\mathcal P(U)$-marked set
$\mathcal G\subset  \Bbbk[C][\mathbf x]_{\mathbf d}^m$ consisting of all
elements
\begin{equation}\label{polymarkKC}
F_\alpha^k=\biggl(x^\alpha-\sum_{x^{\eta}\in \mathcal N(J^{(k)})_{\vert\alpha\vert}}C_{\alpha\eta k k}x^\eta \biggr)e_k
-\mathop{\sum_{l\neq k,x^\eta e_l\in \mathcal N(J^{(l)})e_l}}_{\deg(x^\eta e_l)=\deg(x^\alpha e_k)} C_{\alpha\eta k l}x^\eta e_l
\end{equation}
with $x^{\alpha}e_{k}\in\mathcal P(U)$.  Then, we compute all the complete
reductions $x_i F_{\alpha}^{k} \xrightarrow{{\mathcal G^{(s)}}}_\ast L$ for
every term $x^\alpha e_k \in \mathcal P(U)$ and every non-multiplicative
variable $x_i\in \nmult{F_\alpha^k}$ and collect the coefficients of the
monomials $x^\eta e_j \in \cN(U)$ of all the reduced elements $L$ in a set
$\mathcal R\subset  \Bbbk[C]$.

\begin{theorem}\label{thm:rappr}
The functor   $\MFFunctor{\mathcal P(U)}^{\, m,  \mathbf d}$   is represented by the scheme $\mathrm{Spec}( \Bbbk[C]/(\mathcal{R}))$, that we denote by $\MFScheme{\mathcal P(U)}^{\, m,  \mathbf d}$.
\end{theorem}

\begin{proof}
We observe that each element $ f_\alpha^k$  of a  $\mathcal P(U)$-marked set   $G$   in $\Sk^m_{\mathbf d}$ can be written in the following form:
\[
f_\alpha^k=\left(x^\alpha-\sum_{x^{\eta}\in \mathcal N(J^{(k)})_{\vert\alpha\vert}}c_{\alpha\eta k k}x^\eta \right)e_k
-\mathop{\sum_{l\neq k,x^\eta e_l\in \mathcal N(J^{(l)})e_l}}_{\deg(x^\eta e_l)=\deg(x^\alpha e_k)} c_{\alpha\eta k l}x^\eta e_l, \quad c_{\alpha\eta k l}\in A.
\]
Therefore, $G$ can be obtained by specializing in $\mathcal G$ the
variables $C_{\alpha\eta k l}$ to the constants $c_{\alpha\eta k l}\in A$.
Moreover, $G$ is a $\mathcal P(U)$-marked basis if and only
$ x_i f^k_{\alpha} \xrightarrow{{ G^{(s)}}}_\ast 0$ for every
$x^\alpha e_k \in \mathcal P(U)$ and $x_i\in\nmult{f_\alpha^k}$.
Equivalently, $G$ is a $\mathcal P(U)$-marked basis, if and only if the
evaluation morphism $\varphi: \Bbbk[C]\rightarrow A$,
$\varphi(C_{\alpha\eta k l})=c_{\alpha\eta k l}$, factors through
$ \Bbbk[C]/(\mathcal R)$, namely, if and only if the following
diagram commutes
\[
\xymatrix{
 \Bbbk[C]  \ar[r]^\varphi \ar[d]& A\\
 \Bbbk[C]/(\mathcal R)\ar[ru]}\,.
\]
\end{proof}

\begin{remark}
  The arguments presented in the proof of Theorem \ref{thm:rappr}
  generalize those presented in \cite{CMR15, LR2} for ideals to our more
  general framework of modules.

  As a consequence of this result we know that the scheme defined as
  $\mathrm{Spec}( \Bbbk[C]/(\mathcal R))$ only depends on the
  submodule $U$ and not on the possibly different procedures for
  constructing it: any other procedure that gives a set of \lq\lq
  minimal\rq\rq\ conditions on the coefficients $C$ that are necessary and
  sufficient to guarantee that a $\mathcal P(U)$-marked set $G$ is a
  $\mathcal P(U)$-marked basis generates an ideal $\mathcal R'$ such that
  $ \Bbbk[C]/(\mathcal R)= \Bbbk[C]/(\mathcal R')$.
\end{remark}

\section{$\mathcal{P}(U)$-marked Bases and Syzygies}\label{sec:syz}

We now study syzygies of a $\mathcal{P}(U)$-marked basis and we formulate
a $\mathcal{P}(U)$-marked version of the involutive Schreyer theorem
\cite[Theorem 5.10]{Seiler2009II}.  For notational simplicity, this section
is formulated for $m=1$, that is for ideals in $\Sk$, but it is straightforward to extend
everything to submodules $\Sk^m_{\mathbf d}$ generated by a marked basis over a quasi-stable module.

Let $J$ be a quasi-stable monomial ideal in $\Sk$ and $I$ an ideal in
$\Sk$ generated by a $\mathcal{P}(J)$-marked basis $G$. Let $m$ be the
cardinality of $\mathcal P(J)$. We denote the terms in $\mathcal P(J)$ by
$x^{\alpha(k)}$ and the polynomials in $G$ by $f_{\alpha(k)}$, with
$k\in \{1,\dots,m\}$.

\begin{lemma}\label{lem:uniqueStRep}
Every polynomial $f \in I$ can be uniquely written in the form $f = \sum_{l=
1}^m P_{l} f_{\alpha(l)}$ with $f_{\alpha(l)}\in G$ and $P_{l} \in A[\mult{f_
{\alpha(l)}}]$.
\end{lemma}

\begin{proof}
This is a consequence of Corollary \ref{lem:markedlexorder} and Theorem \ref
{th:markedSetChar} \eqref{it:markedSetChar_vi}.

\end{proof}

Take an arbitrary element $f_{\alpha(k)} \in G$ and choose an arbitrary non-multiplicative
variable $x_i \in \nmult{f_{\alpha(k)}}$ of it. We can
determine, via the reduction process $\xrightarrow{G^{(s)}}$, for each
$f_{\alpha(l)}\in G$ a unique polynomial $P_l^{k;i} \in
A[\mult{f_{\alpha(l)}}]$ such that $x_i f_{\alpha(k)} = \sum_{l=1}^m P_l^{k;i}
f_{\alpha(l)}$.  This relation corresponds to the \emph{fundamental syzygy}
\begin{equation*}
S_{k;i} = x_i e_k - \sum_{l=1}^m P_l^{k;i} e_l\;.
\end{equation*}
We denote the set of all fundamental syzygies by
\begin{equation*}
 G_{\Syz} = \{S_{k;i} \; | \; k\in \{1,\dots,m\}, \; x_i \in \nmult{f_{\alpha(k)}}\}\;.
\end{equation*}

\noindent
We consider the syzygies in $G_{\Syz}$ as elements of  $\Sk^m_{\mathbf d}$
with $\mathbf d=(\deg(x^{\alpha(1)}), \dots, \deg(x^{\alpha(m)}))$.

\begin{lemma}\label{lem:zerosyz}
Let $S = \sum_{l=1}^m S_l e_l$ be an arbitrary syzygy of the $\mathcal
P(J)$-marked basis $G$ with coefficients $S_l \in \Sk$. Then $S_l \in
A[\mult{f_{\alpha(l)}}]$ for all $1 \leq l \leq m$ if and only if $S = 0_A^m$.
\end{lemma}
\begin{proof}

If $S \in \Syz(G)$, then $\sum_{l=1}^m S_l f_{\alpha(l)} = 0$. According to
Lemma \ref{lem:uniqueStRep}, each $f \in I$ can be uniquely written in the
form $f = \sum_{l= 1}^m P_{l} f_{\alpha(l)}$ with $f_{\alpha(l)}\in G$ and
$P_{l} \in A[\mult{f_{\alpha(l)}}]$. In particular, this holds for $0_A \in
I$. Thus
$0_A = S_l \in A[\mult{f_{\alpha(l)}}]$ for all $l$ and hence $S = 0_A^m$.
\end{proof}

\begin{lemma}\label{lem:syzmarkedset}
Let $U$ be the monomial module   $U=\oplus_{l=1}^m(\nmult{x^{\alpha(l)}}) e_l$
where   $(\nmult{x^{\alpha(l)}})$ is the ideal generated by
$\nmult{x^{\alpha(l)}}$ in $\Sk$. Then $U$ is a quasi-stable module with
Pommaret basis   $\mathcal{P}(U)= \{x_i e_l \ | \ 1\leq l\leq m, x_i \in
\nmult{f_{\alpha(l)}} \}$ and $G_{\Syz}$ is a $\mathcal{P}(U)$-marked set   in
$\Sk^m_{\mathbf {d}}$.
\end{lemma}
\begin{proof}
By \cite[Lemma 5.9]{Seiler2009II} we can immediately conclude that $U$ is a
quasi-stable module and that the set $\{x_i e_l \ | \ 1\leq l\leq m, x_i \in
\nmult{f_{\alpha(l)}} \}$ is the Pommaret basis of $U$.

We define $\Ht(S_{i; l}) = x_i e_l$ and easily see that $G_{\Syz}$ is a
$\mathcal{P}(U)$-marked-set: by definition of $U$,  every term $x^\mu e_k$ in
$\supp (S_{l; i} - x_i e_l)$ belongs to $\mathcal N(U)$, because $x^\mu \in
\mult{f_{\alpha(k)}}$.
\end{proof}

Observe that for every fundamental syzygy $S_{k;i}\in G_{\Syz}$,
$\mult{S_ {k;i}}=\{x_0,\dots,x_i\}$. As in Section \ref{sec:markedModules}, we
define for every degree $s$ the following set of polynomials in $\langle
G_{\Syz}\rangle$:
\[
G_{\Syz}^{(s)}=\{x^\delta S_{k;i}\mid S_{k;i}\in G_{\Syz},x^\delta \in \mult{S_
{k;i}}, \deg(x^\delta S_{k;i})=s\}.
\]

\begin{lemma}\label{lem:syzgen}
The set $G_{\Syz}^{(s)}$ generates the $A$-module $\Syz(G)_s$ for every $s$.
\end{lemma}
\begin{proof}
Let $S = \sum_{l = 1}^m S_l e_l$ be an arbitrary non-vanishing syzygy in
$\Syz(G)_s$. By Lemma \ref{lem:zerosyz}, there is at least one index $k$ such
that the coefficient $S_{k}$ contains a term $x^\mu$ depending on a non-
multiplicative variable $x_i \in \nmult{f_{\alpha(k)}}$.  Among all such
values of $k$ and $\mu$ we choose the term $x^\mu e_k$ which is
lexicographically maximal. Then, $x^\mu e_k$ belongs to the quasi-stable
module $U$, hence there is $x^\delta S_{k;j}\in G^{(s)}_{\Syz}$ such that
$x^\delta x_j =x^\mu$. We define $S^\prime = S - \lambda x^\delta S_{k; j}$,
where $\lambda \not= 0_A$ is the coefficient of $x^\mu e_k$ in $S$.

 Now we have to show that for every $x^\nu$ which is contained in a term
$\lambda x^\nu e_l$ in $\supp(S^\prime)\cap U$, $x^\nu$ is lexicographically
smaller than $x^\mu$. The terms of $\supp(S)\cap U$ contained in
$\supp(S^\prime)$  are by assumption lexicographically  smaller than $x^\mu
e_k$. Every other term arises from  $x^\delta\sum_{l=1}^m P_l^{(k;j)} e_l$. We
know that $x_j f_{\alpha(k)} = \sum_{l=1}^m P_l^{(k;j)} f_{\alpha(l)}$. In
particular, a term $x^{\nu'}$ in $P_l^{(k;j)}$ is lexicographically smaller
than $x_j$, by Corollary \ref{lem:markedlexorder}. Therefore every term in
$x^\delta \sum_{l=1}^m P_l^{ (k;j)} e_\beta$ is lexicographically smaller than
$x^\delta x_j=x^\mu$. If $S^\prime \neq 0$, again by Lemma  \ref{lem:zerosyz},
we iterate the procedure on a lexicographical maximal term of $S'$ containing
a non-multiplicative variable. Since all new non-multiplicative terms introduced
are lexicographically smaller, the reduction process must stop after a finite
number of steps. As a result we get a representation $S^\prime = \sum_{l=1}^l
S'_l e_l$ such that $S'_l \in A[\mult{f_{\alpha(l)}}]$ for all $1 \leq l \leq
m$. But Lemma \ref{lem:zerosyz} says that this sum must be zero.
\end{proof}

\begin{theorem}[$\mathcal{P}(U)$-marked Schreyer Theorem]\label
{thm:markedschreyer}
Let $G = \{f_{\alpha(1)}, \dots f_{\alpha(m)}\}$ be a $\mathcal{P}(J)$-marked
basis. Then $G_{\Syz}$ is a $\mathcal{P}(U)$-marked basis of $\Syz(G)$, with
$U$ as in Lemma \ref{lem:syzmarkedset}.
\end{theorem}
\begin{proof}
By Lemma \ref{lem:syzmarkedset}, we know that $G_{\Syz}$ is a
$\mathcal{P}(U)$-marked set. By Lemma $\ref{lem:syzgen}$, we know that
$\langle G_{\Syz}^{(s)}\rangle^A=\langle \Syz(G)_s\rangle^A$ and we conclude
by Theorem \ref{th:markedSetChar} \eqref{it:markedSetChar_vi}.
\end{proof}

Iterating this result, we arrive at a (generally non-minimal) free
resolution.  In contrast to the classical Schreyer Theorem for Gr\"obner
bases, we are able to determine the ranks of all appearing free modules
without any further computations.

\begin{theorem}\label{thm:freeRes}
Let $G = \{f_{\alpha(1)}, \dots, f_{\alpha(m)}\}$, $\deg(f_{\alpha(i)})=d_i$,
be a $\mathcal{P}(J)$-marked basis and $I$ the ideal generated by $G$ in
$\Sk$. We denote by $\beta_{0,j}^{(k)}$ the number of terms $x^\alpha \in
\mathcal P(J)$ such that $\deg(x^\alpha)=j$ and $\min (x^\alpha) = x_k$ and
set $D=\min_{x^\alpha \in \mathcal P(J)}\{i \mid x_i=\min(x^\alpha)\}$. Then
$I$ possesses a finite free resolution
\begin{equation}\label{eq:freeres}
0\longrightarrow\bigoplus\Sk(-j)^{r_{n-D,j}}\longrightarrow\cdots
\longrightarrow\bigoplus\Sk(-j)^{r_{1,j}}
\longrightarrow\bigoplus\Sk(-j)^{r_{0,j}}\longrightarrow I \longrightarrow0
\end{equation}
of length $n - D$ where the ranks of the free modules are given by
\begin{equation*}
r_{i,j} = \sum_{k = 1}^{n-i} \binom{n-k}{i}\beta_{0,j-i}^{(k)}.
\end{equation*}
\end{theorem}

\begin{proof}
According to Theorem \ref{thm:markedschreyer}, $G_{\Syz}$ is a
$\mathcal{P}(U)$-marked basis for the module $\Syz_1(I)$, with $U$ as in Lemma
\ref{lem:syzmarkedset}. Applying the theorem again, we can construct a marked
basis of the second syzygy module $\Syz_2(I)$ and so on.  Recall that for
every index $1\leq l\leq m$ and for every non-multiplicative variable $x_k\in
\nmult{f_{\alpha(l)}}$ we have  $\min(\Ht(S_{l;k})) = k >
\min(\Ht(f_{\alpha(l)}))$.  If $D$ is the  index of the minimal variable
appearing in a  head term in $G$, then the  index of the minimal variable
appearing in a head term in $G_{\Syz}$ is $D+ 1$. This observation gives an immediate formula to compute the
length of the resolution \eqref{eq:freeres}.  Furthermore $\deg(S_{k;i}) =
\deg(f_{\alpha(i)})$, e.\,g.\   from the $i$-th to the $(i + 1)$-th module the
degree from the basis element to the corresponding syzygies grows by one.

The ranks of the modules follow from a rather straightforward combinatorial
calculation. Let $\beta_{i,j}^{(k)}$ denote the number of generators of degree
$j$ of the $i$-th syzygy module $\Syz_i(G)$ with minimal variable in the head
term $x_k$ . By definition of the generators $S_{l; k}$, we find
\begin{equation*}
\beta_{i,j}^{(k)} = \sum_{t = 1}^{k - 1} \beta_{i - 1,j - 1}^{(n -t)}
\end{equation*}
as each generator with minimal variable smaller than $k$ and degree $j-1$ in the
marked basis of $\Syz_i(G)$ contributes one generator of minimal variable $k$
and degree $j$ to the marked basis of $\Syz_i(G)$. A simple but lengthy induction
allows us to express $\beta_{i,j}^{(k)}$ in terms of $\beta_{0,j}^{(k)}$:
\begin{equation*}
\beta_{i,j}^{(k)} = \sum_{t = 1}^{k - i}\binom{k-l-1}{i-1}\beta_{0,j -i}^{(t)}
\end{equation*}
Now we are able to compute the ranks of the free modules via
\begin{equation*}
r_{i,j} = \sum_{k = 1}^n\beta_{i,j}^{(k)} = \sum_{k = 1}^{n}\sum_{t = 1}^{k
-i}\binom{k-t-1}{i - 1}\beta_{0,j - i}^{(t)} = \sum_{k = 1}^{n -i} \binom{n -k}
{i}\beta_{0,j-i}^{(k)}.
\end{equation*}
The last equality follows from a classical identity for binomial coefficients.
\end{proof}

\begin{remark}
Observe that the direct summands in the resolution \eqref{eq:freeres}
depend only on the Pommaret basis $\mathcal P(J)$ and not on the ideal
$I$, while the maps in \eqref{eq:freeres} depend on $I$.
\end{remark}

\begin{corollary}\label{cor:bounds}
Let $G$ be a $\mathcal{P}(J)$-marked basis and $I$ the ideal generated by $G$
in $\Sk$. Define $r_{i,j}$ as in Theorem \ref{thm:freeRes} and let $b_{i,j}$
be, as usual, the Betti numbers of $I$. Then
\begin{itemize}
\item $b_{i,j}  \leq r_{i,j}$ for all $i,j$;
\item  $\reg (I) \leq \reg(J)$;
\item $\pdim(I) \leq  \pdim(J)$.
  \end{itemize}
\end{corollary}
\begin{proof}
The three inequalities follow from the free resolution \eqref{eq:freeres} of
$I$, recalling that $\reg(J):= \max_{x^\alpha \in \mathcal P(J)}\{\deg
(x^\alpha)\}$ and $\pdim(J)=  n-\min_{x^\alpha \in \mathcal P(J)}\{i \mid
x_i=\min(x^\alpha)\}$.
\end{proof}

If $G$ is even a Pommaret basis for the reverse lexicographic term order,
i.\,e.\ if $J$ is the leading  ideal of $I$ for this order, then we
obtain the stronger results $\reg (I) = \reg(J)$ and $\pdim(I) = \pdim(J)$
(for other term orders we also get only estimates) \cite[Corollaries 8.13,
9.5]{Seiler2009II}.

\begin{example}\label{ex:nonMinRes}

Let $\Sk =  \Bbbk[x_0, x_1, x_2]$, $J$ the monomial ideal with Pommaret basis
$\mathcal{P}(J) = \{x_2^3, x_2^2x_1, x_2 x_1, x_1 x_0, x_1^2\}$ and $I$ the
polynomial ideal generated by $G = \{g_1, g_2, g_3, g_4, g_5\}$ with
\begin{equation*}
\begin{array}{clcl}
& g_1 = x_2^3\,, \qquad &  & g_2 = x_2^2x_1\,, \\
& g_3 = x_2 x_1 \,, \qquad &  & g_4 = x_1x_0 + x_2^2 \,, \\
& g_5 = x_1^2\,. \qquad &  &
\end{array}
\end{equation*}
One easily checks that $G$ is a $\mathcal{P}(J)$-marked basis.

We explicit compute the multiplicative representations of $x_2 \cdot g_2$,
$x_2 \cdot g_3$, $x_1 \cdot g_4$,   $x_2 \cdot g_4$,  $x_2 \cdot g_5$ which
yield the set of fundamental syzygies $G_{\Syz}=\{S_{2;2}, S_{3;2},
S_{4;1}, S_{4;2}, S_{5;2}\} \subset \Sk^5$:
\begin{equation*}
\begin{array}{clcl}
  &x_2 \cdot g_2 = x_1 \cdot g_1\,,&\qquad &S_{2;2} = x_2 \cdot e_2 - x_1 \cdot e_1\,,      \\
  &x_2 \cdot g_3 =           g_2 \,,   &\qquad &S_{3;2} = x_2 \cdot e_3 -           e_2\,,      \\
  & x_1 \cdot g_4 = x_0 \cdot g_5 + g_2\,,&\qquad &      S_{4;1} = x_1 \cdot e_4 - x_0 \cdot e_5 - e_2\,,\\
  &x_2 \cdot g_4 = x_0 \cdot g_3 + g_1\,,&\qquad &       S_{4;2} = x_2 \cdot e_4 - x_0 \cdot e_3 - e_1\,,\\
 & x_2 \cdot g_5 = x_1 \cdot g_3\,, &\qquad &       S_{5;2} = x_2 \cdot e_5 - x_1 \cdot e_3\,.
\end{array}
\end{equation*}

The only non-multiplicative variable for $G_{\Syz}$ is
$\nmult{S_{4;1}} = \{x_2\}$.

Therefore we have to compute the reduction of $x_2 S_{4;1}$ which is $x_2
S_{4;1} = x_1 S_{4;2} - S_{2;2} - x_0 S_{5;2}$ and hence we get
the set of fundamental syzygies of the first syzygy module $G_{\Syz_2} = \{x_2
e_3 - x_1 e_4 - e_1 - x_0 e_5\} \subset \Sk^5$.

This leads to the following free resolution of $I$ of length two:
\begin{multline*}
      0\longrightarrow \Sk(-4) \xrightarrow{\quad \delta_2 \quad} \Sk(-4)\oplus \Sk(-3)^4
      \xrightarrow{\quad \delta_1 \quad} \\
      \xrightarrow{\quad \delta_1 \quad}\Sk(-3)^2\oplus\Sk(-2)^3
      \xrightarrow{\quad \delta_0 \quad} I \longrightarrow 0,
    \end{multline*}
where
    \begin{equation*}
      \delta_0 = \left(
        \begin{array}{ccccc}
          x_2^3 & x_2^2x_1 & x_2 x_1 & x_1x_0 + x_2^2 & x_1^2
        \end{array}
      \right),
    \end{equation*}
    \begin{equation*}
      \delta_1 = \left(
        \begin{array}{cccccc}
          -x_1 &  0   &  0   & -1   &  0   \\
           x_2 & -1   &  -1  &  0   &  0   \\
             0 &  x_2 &  0   & -x_0 & -x_1 \\
             0 &  0   &  x_1 &  x_2 &  0   \\
             0 &  0   & -x_0 &  0   &  x_2 \\
        \end{array}
      \right),\quad
      \delta_2 = \left(
        \begin{array}{c}
             1 \\
             0 \\
           x_2 \\
          -x_1 \\
           x_0
        \end{array}
      \right).
    \end{equation*}

  This free resolution is not minimal,  it is obvious looking at the matrices, which contain constant entries. Minimizing the resolution
  leads to the minimal free resolution of $I$ of length one:
    \begin{equation*}
      0\longrightarrow  \Sk(-3)^2
      \xrightarrow{\quad \delta_1^\prime \quad} \Sk(-2)^3
      \xrightarrow{\quad \delta_0^\prime \quad} I\longrightarrow 0.
    \end{equation*}

    Hence in the present example, we have $1=\pdim(I)<\pdim(J)=2$ and $2=\reg(I)<\reg(J)=3$.

\end{example}

\begin{example}
Let $\Sk =  \Bbbk[x_0, x_1, x_2]$, $J$ the monomial ideal with Pommaret basis
$\mathcal{P}(J)=\{x_2x_1, x_2^2 x_1, x_2^3,x_1^3, x_2^2x_0, x_1^2x_0\}$ and
$I$ be the ideal generated by the $\mathcal P(J)$-marked basis
$G=\{g_1, g_2, g_3, g_4, g_5, g_6\}$ with
\begin{equation*}
\begin{array}{clcl}
& g_1=x_2x_1-x_2^2-x_1^2\,, & \qquad & g_2=x_2^2 x_1\,, \\
& g_3=x_2^3\,, & \qquad & g_4=x_1^3\,, \\
& g_5=x_2^2x_0\,, & \qquad & g_6=x_1^2x_0\,,
\end{array}
\end{equation*}
where $\Ht(g_1)=x_2x_1$. Observe that $G$ is not a Gr\"obner
basis, for any term order, due to the terms in $x_2x_1-g_1$.

By Theorem \ref{thm:freeRes}, we construct the following free resolution of $I$:
\begin{multline}\label{eq:exrisnonmin}
  0\longrightarrow \Sk(-5)^2 \xrightarrow{\quad \delta_2 \quad} \Sk(-3)\oplus\Sk(-4)^6
  \xrightarrow{\quad \delta_1 \quad}\\
  \xrightarrow{\quad \delta_1 \quad} \Sk(-2)\oplus\Sk(-3)^5
  \xrightarrow{\quad \delta_0 \quad} I \longrightarrow 0.
\end{multline}

The above resolution is obviously not minimal. The minimal free resolution of $I$ is
\begin{equation*}
      0\longrightarrow \Sk(-5)^2 \xrightarrow{\quad \delta_2' \quad} \Sk(-4)^6
      \xrightarrow{\quad \delta_1' \quad} \Sk(-2)\oplus\Sk(-3)^4
      \xrightarrow{\quad \delta_0' \quad} I \longrightarrow 0.
    \end{equation*}

In this case, although the resolution \eqref{eq:exrisnonmin} is not minimal, the bounds on projective dimension and regularity given in Corollary \ref{cor:bounds} are sharp.
\end{example}

We will apply the theory of marked bases and schemes to the study of 	Quot schemes. Theorem \ref{thm:freeRes} and Corollary \ref{cor:bounds} seem to suggest that marked bases are particularly suitable to study loci of a Quot scheme given by bounds on the invariants of a module coming from the minimal free resolution: regularity, projective dimension, extremal Betti numbers. However, this is only partly true, due to the fact that in order to study special loci of marked schemes we need to prove inequailities like those in Corollary \ref{cor:bounds} on saturated ideals.

Nevertheless, we will be able to study the locus of a Quot scheme given by an upper bound for the regularity, thanks to the following theorem (and corollary).
The proofs are given for ideals, but they also hold for modules in $\Sk^m$ generated by a marked basis over a quasi-stable module.
\begin{theorem}
Let $J\subseteq \Sk$ be a stable ideal, generated in a single degree $s$, and $I$ be the ideal generated by a $\mathcal P(J)$-marked basis $G$. Then $J$ and $I$ have the same Betti numbers.
\end{theorem}

\begin{proof}
Since $J$ is stable, $\mathcal P(J)$ is the minimal monomial generating set of $J$. Then we can  follow the lines of the proof of \cite[Theorem 4.4]{BBR}, thanks to Theorem \ref{th:markedSetChar}, items \eqref{it:markedSetChar_v}, \eqref{it:markedSetChar_vi}.
\end{proof}

\begin{corollary}\label{cor:RegBoundInMS}
Let $J\subseteq \Sk$ be a stable ideal, generated in a single degree $s$, and $I$ be the ideal generated by a $\mathcal P(J)$-marked basis $G$. Then $\reg(J^{\sat})\geq \reg(I^{\sat})$.
\end{corollary}

%

\section{Deterministic computations for stable positions}\label{sec:detchange}

In the present paper, from Definition \ref{def:MarkedSet} on, we considered
marked sets over a quasi-stable monomial module.  We now focus on marked
sets whose polynomials are generated in a single degree, which is the case
of interest for our applications. We are interested in investigating how to
modify a finite set of polynomials, so that they become a marked set over a
quasi-stable module.

\begin{remark}
  Consider a monomial module $U$ generated by
  $T = \{x^{\mu^{(1)}} e_{k_1}, \dots, x^{\mu^{(q)}} e_{k_q}\}$. Let $s$ be
  the maximal degree of a term in $T$ and assume that $U$ is not
  quasi-stable, i.\,e.\ there exists $x^\mu e_k \in T$ and
  $j > c := \min(x^\mu)$ such that
  $\displaystyle x_j^s \frac{x^\mu}{x_c^{\mu_c}} e_k \notin U$. This
  implies that the term
  $\displaystyle x_j^ {\mu_c} \frac{x^\mu}{x_c^{\mu_c}} e_k$ does not
  belong to $U$. If we now consider the module $\widehat{U}$ generated by
  $\widehat{T}= \{{\displaystyle x_j^ {\mu_c} \frac{x^\mu}{x_c^{\mu_c}}
    e_k},x^{\mu^{(1)}} e_{k_1}, \dots, x^{\mu^{ (q)}} e_{k_q}\}$, then it
  is clear that $\widehat{U}$ is somehow nearer to quasi-stability than
  $U$. This observation is studied in much more detail for the case of
  ideals in \cite{Schweinfurter2016} and \cite{hss:detgen}.
\end{remark}

With the knowledge of the remark above, we define an \emph{elementary move}
$m_{l,t,a}$ as a linear change of variables of the form $x_i \mapsto x_i$
if $i\neq l$ and $x_l \mapsto x_l + a \cdot x_t$ for suitable indices
$l<t $ and a parameter $a \in \Bbbk^{\times}$. If we apply $m_{l,t,a}$ to a term
$x^{\mu} $ we obtain a polynomial
\[
 m_{l,t,a}(x^\mu) = \sum_{i=0}^{\mu_{l}}{\mu_l\choose i} a^i x^\mu\frac{x_t^i}{x_l^i}
\]
The polynomial $m_{l,t,a}(x^\mu)$ always contains at least two terms:
$x^\mu$ with coefficient $1$ and $\frac{x^{\mu}x_t^{\mu_l}}{x_l^{\mu_l}}$
with coefficient $a^{\mu_l}$. In the case of a prime characteristic, any
other coefficient may vanish for some values of $\mu_{l}$ and $j$.  We
extend the linear transformation $m_{l,t,a}$ to polynomials and sets of
polynomials in the obvious way.

It is clear that any monomial module is marked on itself. If we apply a
coordinate transformation on a monomial module, then the new module is
generally no longer monomial, but will have a non-monomial minimal
generating set. The next proposition shows that we can construct again a
marked set out of the new module.

\begin{proposition}\label{prop:cordtransformminor}
  For a given degree $s\geq 0$, let
  $T = \{x^{\mu_{(1)}} e_{k_1}, \dots , x^{\mu_{(q)}} e_{k_q}\}$ be a set
  of terms in $\T^m_s$ and let $\Kext$ be a field extension of $\Bbbk$ such
  that $|\Kext| > sq$. Furthermore, let
  $F = \{f_1, \dots, f_q\} \subset \Kext[\mathbf x]_s^m$ be a $T$-marked
  set.  Assume that $x^\mu e_k := x^{\mu_{(1)}} e_{k_1}$ is an obstruction
  to quasi-stability for $\langle T \rangle$ and set
  $\widehat{F} = m_{c,j,a}(F)$ for an arbitrary $a \in \Kext^{\times}$ and
  some $j>c=\min(x^{\mu})$. Setting
  $x^{\widehat{\mu}} e_k :=\displaystyle
  x_j^{\mu_c}\frac{x^\mu}{x_c^{\mu_c}} e_k \notin \langle T\rangle$, we
  denote by $\widehat{T}$ the set of terms
  $\{x^{\widehat {\mu}} e_k, x^{\mu_{(2)}} e_{k_2}, \dots , x^{\mu_{(q)}}
  e_{k_q}\}$ obtained by replacing the first generator by
  $x^{\widehat{\mu}} e_k$.  Then there exists a set
  $F^\prime \subseteq \langle \widehat{F} \rangle_s$, which is marked over
  $\widehat{T}$ and which can be constructed from $\widehat{F}$ via linear
  combinations.
\end{proposition}

\begin{proof}
  We consider the linear transformation $m_{c,j,a}$ for some
  $a\in \Kext^\times$. The considered term $x^{\mu} e_k$ transforms as
  follows
  \begin{displaymath}
    m_{c,j,a}(x^\mu) e_k =
    \sum_{i=0}^{\mu_{c}}{\mu_c\choose i} a^i x^\mu\frac{x_j^i}{x_c^i}e_{k}\;.
  \end{displaymath}
  By our choice of the index pair $(c,j)$, the term $x^{\widehat\mu} e_k$
  appears on the right hand side with a non-zero coefficient (otherwise the
  considered elementary move was not admissable) for the index value
  $i=\mu_c$.

  Applying the transformation $m_{c,j,a}$ to all polynomials $f_{i}\in F$
  yields new generators $\widehat{f}_{i}$ and each $\widehat{f}_{i}$ still
  contains the term $x^{\mu_{(i)}} e_{k_i}$
  with a coefficient which is a polynomial in $a$ with constant term $1$.
  It may happen that the term
  $x^{\mu_{(i)}} e_{k_i}$ now also appears in
  other generators $\widehat{f}_{l}$, but then its coefficient there is
  always a polynomial in $a$ without a constant term.  Furthermore, in
  $\widehat{f}_{1}$ the term $x^{\widehat{\mu}} e_k$ now appears. Its
  coefficient contains in particular the term $a^{\mu_c}$ coming from the
  above transformation of $x^{\mu}$.  If $x^{\widehat{\mu}} e_k$ also lies
  in the support of some other generator $\widehat{f}_{l}$, then its
  coefficient cannot contain the term $a^{\mu_c}$, as $x^{\mu} e_k$
  appeared only in $f_{1}$, as the set $F$ was assumed to be marked over
  $T$.

  These observations imply that, after taking suitable linear combinations
  of the polynomials $\widehat{f}_i$, we can arrive at a set of polynomials
  $F^\prime:=\{\widehat{h}_{1},\dots,\widehat{h}_{q}\}$ such that for
  $i\in\{1,\dots,q\}$ the only term of $\widehat{T}$ appearing in
  $\widehat{h}_i$ is: $x^{\widehat{\mu}}e_{k}$ for $i=1$, and
  $x^{\mu_{(i)}}e_{k_i}$ for the other values
  of $i$.

  It cannot happen that for some $i = 2,\dots, q$ the term
  $x^{\mu_{(i)}}e_{k_i}$ vanishes when we
  perform the linear combinations on
  $\widehat{h}_{1},\dots,\widehat{h}_{q}$, because there is exactly one
  term $x^{\mu_{(i)}}e_{k_i}$ which has as
  coefficient a polynomial in $a$ with constant term $1$. By the same
  argument, it is clear that the term $x^{\widehat{\mu}}e_{k}$ does not
  vanish by performing linear combinations as its coefficient $a^{\mu_c}$
  in $\widehat{h_1}$ is unique. But this implies that the set $F^\prime$
  obtained from $F$ by the coordinate transformation and suitable linear
  combinations is marked over $\widehat{T}$.  Furthermore, in each
  polynomial $\widehat{h_i}$ the coefficient of the head module term is a
  polynomial in $a$ of degree at most $s$. Since we have $q$ such
  coefficients, the assumption $\vert \Kext\vert> sq$ guarantees that there
  exists a choice for $a$ such that none of these polynomials vanishes.
\end{proof}

From now on, we will assume for simplicity that the field $\Bbbk$ is
infinite, hence we will use coordinate transformations in
$\PGL:=\PGL_{\Bbbk}(n+1)$.  For every element
$g \in \PGL := \mathrm{\PGL}_{\Bbbk}(n + 1)$, $\tilde{g}$ denotes the
automorphism induced by $g$ on $\Sk^m$ and and $g\centerdot$ denotes the
corresponding action on an element. If $G$ is a subset of $\Sk^m$,
$\tilde g \centerdot G$ is the set obtained by applying $\tilde g$ to every
element of $G$.  We can now rephrase Proposition
\ref{prop:cordtransformminor} in the following way, keeping in mind that
under the hypothesis that $\Bbbk$ is infinite, it is also Zariski dense in
any field extension $\Kext$.

\begin{corollary}
  Let $F\subset \Sk^m_s$ be a finite set of polynomials. Then there exists
  a transformation $g\in \PGL$ such that $\tilde g\centerdot F$ is a marked
  set over a quasi-stable module.
\end{corollary}

\begin{lemma}\label{lem:existsQS}
  Consider $\ell\geq 0$. Let $F$ be a saturated module in
  $\Kext[\mathbf x]^m$ for a field extension $\Kext$ of $\Bbbk$ with
  Hilbert polynomial $p(z)$ and $\reg(F)\leq \ell$. Then there exists a
  transformation $g\in \PGL$ and a stable module
  $U=\langle U_\ell\rangle\subset \Kext[\mathbf x]^m$ having Hilbert
  polynomial $p(z)$ and $\reg(U^{\sat})\leq \ell$ such that
  $g\circ \langle F_\ell\rangle$ belongs to $\MFFunctor{U}^m$
\end{lemma}

\begin{proof}
  By \cite[Thm.~2.16]{Seiler2009II} (or \cite[Thm.~6.11,
  Rem.~6.13]{hss:detgen}),\footnote{These references consider only the case
    of ideals. However, the extension to modules along the lines of
    Prop.~\ref{prop:cordtransformminor} is straightforward.} there exists a
  transformation $g \in \PGL$ such that $g\circ \langle F_\ell\rangle$ has
  a Pommaret basis for the degree reverse lexicographic term order. This
  means that the initial module $U$ of $g\circ \langle F_\ell\rangle$ is
  quasi-stable and has the same Hilbert polynomial as $F$. Since
  $\reg(U)=\reg(F_\ell)=\ell$, we have that $U=\langle U_\ell\rangle$ and
  $U$ is even stable. Now it suffices to observe that
  $\reg(U^{\sat})\leq \reg(U)=\ell$.
\end{proof}

\section{Definition of Quot functor and Quot functor with bounded regularity}\label{sec:defquot}

From now on, all the modules $M\in \Sk^m$, with $A$ a $\Bbbk$-algebra

Let $p(z)\in \mathbb Q[z]$ be the Hilbert polynomial of $\Kx^m/M$, for some homogeneous module $M\subseteq \Kx^m$. We denote by $N_m(z)$ the polynomial $m{n+z\choose z}$ and by $q(z)$  the polynomial $N_m(z)-p(z)$.

By \cite[Proposition 3.1]{Dell}, there is a unique Gotzmann representation of $p(z)$:
\[
p(z)={z+a_1 \choose a_1}+{z+a_2-1 \choose a_2}+\cdots+{z+a_r-(r-1) \choose a_r}
\]
where $a_1\geq a_2\geq \cdots a_r\geq 0$.

We define $r$ as the Gotzmann number of $p(z)$. We recall that, by \cite[Proposition 4.1]{Dell} $r$ is an upper bound for the regularity of the associated sheaf $\widetilde M$.

We now define the Hilbert function and the Hilbert polynomial in a more general case, following the lines of \cite{nitsure}:

Let $X$ be a finite type scheme over a field $ \Bbbk$, together with a line
bundle $L$. Recall that if $F$ is a coherent sheaf on $X$ whose support is
proper over $ \Bbbk$, then the \emph{Hilbert polynomial} $\Phi \in \QQ[z]$ of
$F$ is defined by the function
\[
\Phi(z) = \chi(F(z)) = \sum_{i=0}^n (-1)^i \dim_ \Bbbk H^i(X, F \otimes L^
{\otimes z})
\]
where the dimensions of the cohomologies are finite because of the coherence and
properness conditions. The fact that $\chi(F(m))$ is indeed a polynomial in $m$
under the above assumption is a special case of what is known as Snapper's Lemma
(see \cite[Theorem B.7]{KleimanFGA} for a proof).

Let $X \longrightarrow S$ be a finite type morphism of noetherian schemes,
and let $L$ be a line bundle on $X$. Let $\mathcal{F}$ be any coherent sheaf on
$X$ whose schematic support is proper over $S$. Then for each $s \in S$ we get a
polynomial $\Phi_s \in \QQ[z]$ which is the Hilbert polynomial of the
restriction $\mathcal{F}_s = F|_{X_s}$ of $\mathcal{F}$ to the fiber $X_s$ over
$s$, calculated with respect to the line bundle $L_s = L|_{X_s}$. If $\mathcal
{F}$ is flat over $S$ then the function $s \mapsto{\Phi_s}$ from the set of
points to $S$ to the polynomial ring $\QQ[z]$ is known to be locally constant on
$S$

We will denote by $\PP^n$ the projective space over $ \Bbbk$.
If $Z$ is a $ \Bbbk$-scheme we define $\PP^n_Z := \PP^n \times_ \Bbbk Z$ and if $A$ is a finitely generated $ \Bbbk$-algebra, $\PP^n_A$ is defined as $\PP^n_{\Spec (A)}$.

We are interested in the case $X =\PP^n_Z$ and $L = \mathcal{O}_{\PP^n_Z}(1)$ and $F$ is a quotient of $ \mathcal O^m_{\PP^n_Z}$.

If $F$ is flat over $Z$ then the Hilbert polynomial of the fibres is locally constant. If it is constant we call it the Hilbert polynomial of $F$.

In the following, $\QuotFunctor$ will denote the Quot functor $\schm^{\circ}\rightarrow \Sets$
that associates to an object $Z$ of the category of schemes over $ \Bbbk$ the set
\begin{equation*}
 \QuotFunctor(Z) = \{ \mathcal Q \hbox{ quotients   of }  \mathcal O^m_{\PP^n_Z}\  \text{ flat  over } Z \hbox{ with  Hilbert
 polynomial } p(z)\}.
\end{equation*}

The Hilbert polynomial of $\mathcal Q$ is defined considering the Hilbert polynomial of each fibre of $Z$
and to any morphism of schemes $\varphi \colon Z \rightarrow Z'$ the map
\[
\begin{split}
\QuotFunctor(\varphi )\colon&\ \QuotFunctor(Z') \rightarrow \QuotFunctor(Z)\\
&\parbox{2.18cm}{\centering $\mathcal Q'$} \mapsto\ \varphi^\ast\mathcal Q'
\end{split}
\]

The Quot functor was introduced by Grothendieck in \cite{GroHilbert}, where he also proved that this functor is the functor of points of a projective scheme. In the present paper, we will not use this fact, but we will give an independent proof of the existence of the Quot scheme. Here, we only assume that the Quot functor is a Zariski sheaf \cite[Section 5.1.3]{nitsure};
hence, we can consider it as a covariant functor from the category
of noetherian  $ \Bbbk$-algebras \cite[Lemma E.11]{Sernesi}
\begin{equation*}
\QuotFunctor\colon \Kalg \rightarrow \Sets
\end{equation*}
such that for every finitely generated  $ \Bbbk$-algebra $A$
\[
\QuotFunctor(A) = \left\{ \mathcal Q \hbox{ quotients   of } \ \mathcal O^m_{\PP^n_A} \ \hbox{ flat over } \Spec A \text{ flat  with  Hilbert
polynomial }
 p(z) \right\}.
\]
and for any $ \Bbbk$-algebra  morphism $f\colon A \rightarrow B$
\[
\begin{array}{rccl}
\QuotFunctor(f)\colon&\ \QuotFunctor(A)\ &\rightarrow& \QuotFunctor(B)\\
&\widetilde{ Q }& \mapsto\ &\widetilde{Q\otimes_A    B }
\end{array}
\]
where $Q=H^0_\ast\mathcal Q$, for $\mathcal Q\in \QuotFunctor(A) $.

This is equivalent to consider the functor $\Kalg \rightarrow \Sets$ that associate to every $ \Bbbk$-algebra $A$ the set
\begin{equation*}
 \QuotFunctor(A) = \{  M  \ \hbox{saturated submodules   of } \Sk^m \text{ s.t. } \Sk^m/M \hbox{ flat with  Hilbert  polynomial } p(z)\}.
\end{equation*}
and to every $ \Bbbk$-algebras homomorphism $f\colon A \rightarrow B$ the function
\[
\begin{array}{rccl}
\QuotFunctor(f)\colon&\ \QuotFunctor(A)\ &\rightarrow& \QuotFunctor(B)\\
& { M }& \mapsto\ & {M\otimes_A    B }
\end{array}
\]

Inspired by the results in Section \ref{sec:syz} and by \cite{BBR} for Hilbert schemes, we intend to study a special subfunctor of the Quot functor, defined by giving an upper bound on the Castelnuovo-Mumford regularity of the elements in $\QuotFunctor(A)$, for any $ \Bbbk$-algebra $A$.  Several  proofs use the same arguments of corresponding results in \cite{BBR}.

For every saturated module $M\in \QuotFunctor(A)$, if $A$ is local we define its Castelnuovo-Mumford regularity $\reg(M)$ in the obvious way; otherwise we say that the Castelnuovo-Mumford regularity of $M$ is $\min\{\reg(M\otimes_A A_{\mathfrak p}) \vert \mathfrak p \text{ prime ideal in } A\}$.

\begin{definition}\label{def:BoundedQuot} Let $\ell$ be an integer.
The \emph{Quot functor with bounded regularity}, that we denote by $\QuotFunctorReg{\ell}$, is  the subfunctor of  $\QuotFunctor$ that associates to every Noetherian $ \Bbbk$-algebra $A$ the set $\QuotFunctorReg{\ell}(A)=\{M\in \QuotFunctor\vert \mathfrak \reg(M)\leq \ell\}$. 
\end{definition}
 
   It is immediate that if $\ell'\leq \ell$, then for every $ \Bbbk$-algebra $A$,  $\QuotFunctorReg{\ell'}(A)$ is a subset of $\QuotFunctorReg{\ell}(A)$. Furthermore, if $r$ is the Gotzmann number of $p(z)$, $\QuotFunctorReg{r}$ is exactly $\QuotFunctor$.

From now on, we set two positive integers, $\ell$ and $s\geq \ell$. 
For every $ \Bbbk$-algebra $A$, for every $M\in \QuotFunctorReg{\ell}(A)$, there is a unique graded $\Sk$-module generated in degree $s$, whose saturation is $M$:  $\langle M_{s}\rangle$. Hence,  the Quot Functor with bounded regularity can be considered by \cite[Lemma 5.2 and Theorem 5.1]{Dell} as a subfunctor of the following  Grassmann functor:
\[
\begin{array}{rcll}
\GrassFunctor{s}:  \Kalg &\rightarrow &\Sets &\quad\quad \text{ with }N_m(s)=m{n+s\choose s} \\
A&\mapsto &\GrassFunctor{s}(A)
\end{array}
\]
where
\[
\GrassFunctor{s}(A)=\{A\text{-submodule }F \subseteq \Sk^m_s\text{ such that }
\Sk^m_s/ F\text{ is locally free of rank }p(s)
\}.
\]

Therefore, the the Quot Functor with bounded regularity can be seen as a   subfunctor of the Grassmann functor $\GrassFunctor{s}$ in the following way
\begin{equation*}
 \QuotFunctorReg{\ell}(A) = \{  F\in  \GrassFunctor{s}(A) \hbox{ with } \Sk^m/\langle F \rangle  \hbox{ flat with  Hilbert  polynomial } p(z) \text{ and } \reg(F^{\sat})\leq \ell\}.
\end{equation*}
and to every $ \Bbbk$-algebras omomorphism $f\colon A \rightarrow B$ the function
\[
\begin{array}{rccl}
\QuotFunctorReg{\ell}(f)\colon&\ \QuotFunctorReg{\ell}(A)\ &\rightarrow& \QuotFunctorReg{\ell}(B)\\
& { F }& \mapsto\ & {F\otimes_A    B }
\end{array}
\]

Furthermore we define the natural transformation of functors
\[
\mathcal{H}^{[s]}: \QuotFunctorReg{\ell}\rightarrow \GrassFunctor{s}\,.
\]

We denote by $\pi_M$ the canonical projection $\Sk^m_s\rightarrow \Sk^m_s/M_s$.
 The functor $\GrassFunctor{s}$ is representable and  the representing scheme $\GrassScheme{s}$ is called the Grassmannian. By Plucker embedding, it can be seen as a closed subscheme of $\mathbb P^{\binom{N_m(s)}{p(s)}-1}$.

We will now introduce some useful subfunctors of $\GrassFunctor{s}$ and $\QuotFunctorReg{\ell}$.

We set a basis $\{b_1,\dots, b_{p(s)}\}$ for $A^{p(s)}$. Consider the complete list $\T_s^m=\{\tau_\ell\}_{\ell=1,\dots,N_m(s)}$ of terms $\tau=x^\alpha e_i$, $\vert \alpha\vert=s$, of $ \Bbbk[\mathbf x]_s^m$. $\T_s^m$ is the basis we consider for the $A$-module $\Sk^m_s$. 
 For every element $g \in \PGL := \mathrm{\PGL}_\QQ(n + 1)$, we denote by $\tilde{g}$ also the automorphism induced by $g$ on the Grassmann and Quot
functors
and $g\centerdot$ denotes the corresponding
action on an element.

Consider  $\mathcal I=\{a_1,\dots, a_{p(s)}\}\subset\{1,\dots, N_m(s)\}$, $\vert \mathcal I\vert =p(s)$ and
consider the injective morphism
$\Gamma_{\mathcal I} : A^{p(s)}\rightarrow \Sk^{m}_s,\quad
b_i\mapsto \tau_{a_i}$,
$g \in \PGL$ and the subfunctor $\GrassFunctorL{\mathcal I,g}{s}$
that associates to every noetherian $ \Bbbk$-algebra $A$ the set
\[\GrassFunctorL{\mathcal I,g}{s}(A)=\{F\in \GrassFunctor{s}(A)\vert \pi_F\circ  \tilde{g} \circ\Gamma_{\mathcal I}\text{ is surjective}\}\]

The open subfunctors $\GrassFunctorL{\mathcal I, \mathrm{Id}}{s}$ provide an open cover of $\GrassFunctor{s}$, as $\mathcal I$ varies among the subsets of $\{1,\dots N_m(s)\}$ containing $p(s)$ elements \cite[Lemma 8.13]{GW}. We refer to these open subfunctors as \emph{standard open
cover} of  $\GrassFunctor{s}$.

For every $\mathcal I\subset\{1,\dots, N_m(s)\}$, $\vert \mathcal I\vert=p(s)$, for every $g\in \PGL$, we define the following open subfunctors of 
$\QuotFunctorReg{\ell}$:
\begin{equation}\label{eq:standardcover}
 \QRegA{\mathcal I,g}{\ell,s}(A):=\left(\mathcal{H}^{[s]}\right)^{-1}\left(\GrassFunctorL{\mathcal I,g}{s}(A)\right)\cap \QuotFunctorReg{\ell}.
\end{equation}
 Obviously, taking $g=\mathrm{Id}$, as $\mathcal I$ varies, the subfunctors in \eqref{eq:standardcover} cover $\QuotFunctorReg{\ell}$.

\section{Quasi-stable open cover of the Grassmannian}\label{sec:covergr}

We can associate to the set $\mathcal I\subset\{1,\dots, N_m(s)\}$  the  set of monomials  $\mathcal U_{\mathcal I} := \{\tau_i\}_{i \in \mathcal I}\subset
\T_s^m$  and its complementary $\UC$. Observe that if $\vert \mathcal I\vert=p(s)$, then $\vert \UC\vert=N_m(s)-p(s)$. In this paper we prefer to consider a different open cover of
the Quot scheme defined considering some special $\mathcal U_{\mathcal I}$.

\begin{lemma}\label{lem:grassFunctorUmarkedset}
Consider  $\mathcal I=\{a_1,\dots, a_{p(s)}\}\subset\{1,\dots, N_m(s)\}$, $\vert \mathcal I\vert =p(s)$. 
Let us assume that the monomial module $U:=\langle \UC\rangle\subset \Sk^m$is quasi-stable.
\begin{enumerate}[(i)]
\item  $F\in \GrassFunctorL{\mathcal I,\mathrm{Id}}{s}(A)$ if and only if it is generated as an $A$-module by a $\UC$-marked set.
\item If $F$ belongs to $\GrassFunctorL{\mathcal I,\mathrm{Id}}{s}(A)$, then for every
$s^\prime\geq s$ the $A$-module    $\langle F\rangle_{s^\prime}$   contains a
free submodule of rank $\geq q(s^\prime)$ generated by a ${\langle \UC \rangle
}\cap \T_{s^\prime}^m$-marked set.
\end{enumerate}
\end{lemma}

\begin{proof}\ 

\begin{enumerate}[(i)]
\item If $F$ belongs to $\GrassFunctorL{\mathcal I,\mathrm{Id}}{s}(A)$, since $\pi_F\circ \mathrm{Id} \circ \Gamma_{\mathcal I}$ is surjective, $\mathcal U_{\mathcal I}$ is a generating set for the module $\Sk^m_s/F$. Then, for every $\tau\in \UC$, we consider the polynomial $f_\tau=\tau-\pi_F(\Gamma_{\mathcal I}(\tau))$. The module element $f_\tau$ is a homogeneous marked element of $\Sk^m$, with $\Ht(f_\tau)=\tau$ and $\tau-f_\tau \in \langle \mathcal U_{\mathcal I}\rangle^A=\langle\mathcal N(U)_s\rangle^A$. Hence $\{f_\tau\}_{\tau\in \UC}$ is a $\UC$-marked set contained in $F$. Observe that $\langle f_\tau\rangle^A\subset F$ and $\rk(F_s)=\rk\langle f_\tau\rangle$, hence $F=\langle f_\tau \rangle^F$.

Vice versa, let $G=\{f_\tau\}_{\tau\in \UC}$ be the $\UC$-marked set generating $F$. Then every $\tau \in  \Sk^m${, there is $g\in F_{\vert\tau\vert}$ such that $\tau-f=\sum_{\tau'\in \mathcal N(U)_{\vert\tau\vert}}a'\tau'$, $a'\in A$ by Corollary \ref{lem:markedlexorder}}. 
Hence the $A$-module $F$ generated by $\{ f_\tau\}_{\tau \in \mathcal U^c}$ belongs to $\GrassFunctorL{\mathcal I,\mathrm{Id}}{s}(A)$.
\item We denote by $G^{(s^\prime)}$ the set $\{x^\delta f_\tau\vert f_\tau \in G, \deg(x^\delta f_\tau)=s^\prime, \min(\tau)\geq \max(x^\delta)\}$.
Due to the fact that $\langle G^{(s^\prime)} \rangle^A \subset \langle F \rangle_
{s^\prime}$ this statement follows from Theorem \ref{th:markedSetChar} \eqref{it:markedSetChar_iii}.
\end{enumerate}
\end{proof}

The following example shows that $\Sk^m/\langle \UC\rangle $ and $\Sk^m/\langle F\rangle$, with $F\in \GrassFunctorL{\mathcal I,\mathrm{Id}}{s}(A)$, in general do not have the same Hilbert polynomial or function.
\begin{example}
  In $\Sk=k[x_2, x_1, x_0]$, we consider $\UC=\{x_1x_2, x_0^2\}$ and  $U:=\langle \UC\rangle\subset \Sk$. Let $M$ be the submodule of $\Sk$ generated by $f_1 = x_1x_2+x_0x_1$, $f_2 = x_0^2+x_0x_2$,
which form a $\UC$-marked set. The Hilbert polynomial of $\Sk/U$ is constant, while the Hilbert polynomial of $\Sk/\langle f_1,f_2\rangle$ has degree 1. Hence they also do not have the same Hilbert function.
\end{example}

\begin{lemma}\label{lem:locRing}
Let  $(A, \mathfrak m,  \Kext  )$ be a local ring and $F \in
\GrassFunctor{s}(A)$. Then $F \in \GrassFunctorL{\mathcal I,\mathrm{Id}}{s}(A)$ if
and only if $F \otimes_A    \Kext  \in \GrassFunctorL{\mathcal I,\mathrm{Id}}{s}(  \Kext  )$.
\end{lemma}

\begin{proof}
By the extensions of the scalars it is clear that $F \otimes_A   \Kext  \in \GrassFunctorL{\mathcal I,\mathrm{Id}}{s}(  \Kext  )$ if
$F \in \GrassFunctorL{\mathcal I,\mathrm{Id}}{s}(A)$. Therefore we only
prove the other direction.

Assume that $F \otimes_A   \Kext  \in  \GrassFunctorL{\mathcal I,\mathrm{Id}}{s}(  \Kext  )$ and let
$\{\overline{f}_\tau\}_{\tau\in \UC}$ the $\UC$-marked set
generating $F \otimes_A   \Kext  $.
Let us consider a set of polynomials  $\{{f}_\tau\}_{\tau\in \UC}
\subset F$ such that the image of each $f_\tau$  in $ \Kext  [\mathbf x]^m_s$ is
$\overline{f}_\tau$.

We construct for $F$ a $q(s) \times N_m(s)$ matrix $M_F$. We order (in any way)
the terms of $\T_s^m$: $x^{\alpha_1}e_{k_1},
\dots,x^{\alpha_{N_m(s)}}e_{k_{N_m(s)}}$ and the elements $f_\tau$. The $j$-th
column of $M$ corresponds to the term $x^{\alpha_j}e_{k_j}$. The $i$-th row of
$M_F$ corresponds to the coefficients in the $i$-th element in
$\{{f}_\tau\}_{\tau\in \UC}$.

Considering the images of the entries in $  \Kext  $ we obtain the analogous matrix
$M'$ for   $\{\overline{f}_\tau\}_{\tau\in \UC}$. By hypothesis the
minor corresponding to $\mathcal U^c$ of this last matrix is invertible. Then
the corresponding minor in $M$ is also invertible, because $A$ is local.

In general $\{{f}_\tau\}_{\tau\in \mathcal U^c}$ is not a $\UC$-marked
set. But we can obtain a $\UC$-marked set by performing a row reduction of $M$
such that the minor from above gets the identity matrix.
\end{proof}

\begin{definition}\label{def:stable}
For any admissible Hilbert polynomial $p(z)$ in $\Sk^m$,  given $\ell$, we consider the integers $s\geq \ell$,  $p(s)$, $N_m(s)$. We define the following sets:
\begin{itemize}
\item $\mathbb{QS}$ is the set of the quasi-stable modules in $\kappax^m$ whose minimal monomial set of generators consists of $N_m(s)-p(s)$ terms of degree $s$.
\item $\mathbb{QS}_{p(z)}$ is the subset of  $\mathbb{QS}$  containing monomial modules having Hilbert polynomial $p(z)$.
\item $\mathbb{QS}_{p(z)}^{\ell}$  is the subset of $\mathbb{QS}_{p(z)}$ containing  submodules $U$   with  $\reg(U^{\sat})\leq \ell$.
\item  ${L}^{[\ell,s]}_{p(z)}$ is the closed subset of $\GrassScheme{s}$ defined by the ideal 
$$\left( g\centerdot  \Delta_{\mathcal{I}} \ \vert \ \forall \  g \in \PGL(n+1), \ \forall \
 \langle \UC\rangle \in \mathbb{QS}^{\ell}_{p(z)} \right).$$ 
\end{itemize}
\end{definition}

\begin{proposition}\label{prop:QuasiStableCoverGr}
The collection of subfunctors
\[
\left\{\GrassFunctorL{\mathcal I,g}{s} \mid g \in \PGL,\mathcal I\subset\{1,\dots, N_m(s)\}\; \text{s.t} \; \langle \UC \rangle \in \mathbb{ Q S}\right\}
\]
covers the Grassmann functor $\GrassFunctor{s}$.
\end{proposition}

\begin{proof}
We have to proof that for every $ \Bbbk$-Algebra $A$ and every $F \in
\GrassFunctor{s}(A)$ there exist $\mathcal I\subset\{1,\dots, N_m(s)\}$ with $\langle \UC
\rangle \in \mathbb{ Q S}$ and $g \in \PGL$ such that $F \in
\GrassFunctorL{\mathcal I,g}{s}(A)$ or equivalently such that
$g^{-1}\centerdot F \in \GrassFunctorL{\mathcal I,\mathrm{Id}}{s}(A)$.

As the question is local it is sufficient to consider the case that  the ring $A$
is  local. By Lemma \ref{lem:locRing}, we may assume that $A$ is in
fact a field.

Let $F \in \GrassFunctor{s}(A)$ for a field $A$. Let $\mathcal{J}$ be
the set of subsets of $\T_{s}^m$ of cardinality $N_m(s)-p(s)$.  As in the proof of Lemma \ref{lem:locRing}, we associate the $q(s)\times N_m(s)$ matrix $M_F$ to $F$, considering a set of generators for the module $F$. For every $\mathcal V\in \mathcal J$,  let  $\Delta_{\mathcal V}(M_F)$
be the minor of $M_F$ corresponding  to $\mathcal V \in \mathcal{J}$. 
It is obvious that there is $\mathcal V \in
\mathcal{J}$ such that $\Delta_{\mathcal V}(M_F) \neq 0$. 

If $\langle \mathcal V \rangle \in \mathbb{ Q S}$ we are already done: if $\mathcal I\subset \{1,\dots, N_m(s)\}$ such that $\UC=\mathcal V$, then $F$ belongs to $\GrassFunctorL{\mathcal I,\mathrm{Id}}{s}(A)$ . Assume that this
is not the case. Then there exists an obtruction to  quasi-stability: $x^\mu e_k \in \mathcal V$ and $j >
c := \min(x^\mu)$, such that $x_j \frac{x^\mu}{x_c} e_k \notin
\langle \mathcal V \rangle$. We denote by $\widehat{\mathcal V} \in \mathcal{J}$ the set
obtained by replacing in $\mathcal V$ the obstruction to quasi-stability $x^\mu e_k$
with $x^{\widehat{\mu}} e_{k} := x_j \frac{x^\mu}{x_c} e_k$.

Up to an autoreduction of $F$, we can assume
without loss of generality, that $F$ is generated by a $\mathcal V$-marked set, due to the fact that  $\Delta_{\mathcal V}(M_F)$ is non-zero.  Proposition
\ref{prop:cordtransformminor} guarantees that there is a linear coordinate
transformation $g \in \PGL$ with respect to the elementary move $m_{c,j,a}$ for
an $a \in A$, such that $\widehat{F} = g^ {-1}\centerdot F$ and $\widehat{F}$ is
 generated by a $\widehat{\mathcal V}$-marked set. This implies that
$\Delta_{\widehat{\mathcal V}}(M_{\widehat{F}}) \neq 0$. If $\langle \widehat{\mathcal V}\rangle\in \mathbb{ Q S}$ we are done:  if $\mathcal I\subset \{1,\dots, N_m(s)\}$ such that $\UC=\widehat{\mathcal V}$, then $g^{-1}\centerdot  F$ belongs to  $\GrassFunctorL{\mathcal I,\mathrm{Id}}{s}(A)$.
If $\langle \widehat{\mathcal V}\rangle\notin \mathbb{ Q S}$, we can  repeat this construction starting from an obstruction to stability for the module $\langle \widehat{\mathcal V}\rangle$.

The claim of the proposition follows from a simple termination argument, which
shows, that we finally get a $\widehat{\mathcal V}\in \mathcal J$ and $g\in \PGL$ such that  there is ${\mathcal I}\subset\{1,\dots,N_m(s)\}$ such that $\mathcal U_{{\mathcal I}}^c=\widehat{\mathcal V}$, $\langle \widehat{\mathcal V}\rangle \in \mathbb{ Q S}$ and ${g}^{-1}\centerdot  F$ belongs to $\GrassFunctorL{\mathcal I,\mathrm{Id}}{s}(A)$. We introduce an ordering on $\mathcal {J}$. Given two sets
$\mathcal V_{1},\mathcal V_ {2} \in \mathcal{J}$, we first sort them according to
$\prec_{\mathrm{TOP}_{\mathrm {degrevlex}}}$ (greatest term first) and then compare the two
sets entry by entry again with respect to $\prec_ {\mathrm{TOP}_{\mathrm{degrevlex}}}$. 
Then $\mathcal V_1<\mathcal V_2$ if there is $i$ such that for every $j<i$, the j-th entry of $\mathcal V_1$ is the same as the $j$-th entry of $\mathcal V_2$, while the $i$-th entry of $\mathcal V_1$ is smaller than (or equal to) the $i$-th entry of $\mathcal V_2$ with respect to $\prec_ {\mathrm{TOP}_{\mathrm{degrevlex}}}$.

Our construction gives at each recursion a set $\widehat{\mathcal V}$ such that $\widehat{\mathcal V}>\mathcal V$
with respect to the ordering we defined.  In this way we  construct a strictly
ascending chain of sets in $\mathcal{J}$. Since $\mathcal{J}$ is a finite set,
the chain must be finite, too. Hence, our construction only stops when there are no obstructions to  quasi-stability, that is when it
reaches a set $\widehat{\mathcal V}$ such that $\langle\widehat{\mathcal V}\rangle\in\mathbb{QS}$. 
\end{proof}

\begin{definition}
We call \emph{quasi-stable subfunctor} of $\GrassFunctor{s}$ any
element of the collection of subfunctors of Proposition \ref
{prop:QuasiStableCoverGr}.
\end{definition}
\begin{remark}
The same statement as Proposition \ref{prop:QuasiStableCoverGr} is proved in \cite[Proposition 5.4]{BBR}, concerning the Grassmannian of linear spaces of $\Sk_s$. We highlight that in the present paper we consider  the action of $\PGL$ on $\Sk^m$, hence \cite[Proposition 5.4]{BBR} does not apply this case. Furthermore,  the proof of Proposition \ref{prop:QuasiStableCoverGr} gives an algorithmic strategy to explicitly construct $g$ such that  ${g}^{-1}\centerdot  F$ belongs to $\GrassFunctorL{\mathcal I,\mathrm{Id}}{s}(A)$, with $\langle \UC \rangle \in \mathbb{QS}$.
\end{remark}

\begin{corollary}\label{cor:SchCoverGr}
For every $g \in \PGL$, for every $\mathcal I\subset\{1,\dots, N_m(s)\}$ such that  $\langle \UC \rangle \in \mathbb{QS}$, $\GrassFunctorL{\mathcal I,g}{s}$ is the functor of points of an affine scheme  that we denote by $\mathrm{G}_{\mathcal I,g}^{[s]}$, which is naturally isomorphic to $\mathbb A_ \Bbbk^{(N_m(s)-p(s))\cdot p(s)}$.
\end{corollary}
\begin{proof}
This is analogous to \cite[Proposition 5.4]{BBR}. 
\end{proof}

\begin{proposition}\label{prop:coverGrMinusL}
The collection of  open subschemes
\[
\{\GrassFunctorL{\mathcal I,g}{s}\mid g \in \PGL,\mathcal I\subset\{1,\dots, N_m(s)\}\; \text{s.t} \; \langle \UC \rangle \in 
\mathbb{QS}^\ell_{p(z)}\}
\]
covers $\GrassScheme{s}\setminus {L}^{\ell,s}_{p(z)}$.
\end{proposition}
\begin{proof}
It is sufficient to recall the definitions of $\mathbb{QS}_{p(z)}^{\ell}$ and $L_{p(z)}^{[\ell,s]}$ given in Definition \ref{def:stable}.

\end{proof}

\section{Stable cover and representability of Quot functors}\label{sec:coverQuot}

We will now prove that in order to cover the Quot Functor with bounded regularity $\QuotFunctorReg{\ell}$, it is sufficient to consider $\mathcal I$ such that $\langle \UC\rangle$ is stable, has the same Hilbert polynomial as the modules of the Quot functor and have regularity of their saturation bounded by $\ell$.

We divide the proof in 2 steps. In Proposition \ref{prop:coverQuot} we prove that in order to cover $\QuotFunctorReg{\ell}$ it is sufficient to consider $\mathcal I$ such that $\langle \UC\rangle$ in in $\mathbb{QS}_{p(z)}$. In Theorem \ref{thm:quasiStableCoverQuot}, we prove that only $\mathcal I$ such that $\reg(\langle \UC\rangle^{\sat})\leq \ell$ are necessary and that such a cover actually does not depend on the chosen $s\geq \ell$ for the embedding in the Grassmannian.

\begin{proposition}\label{prop:coverQuot}
Consider $s\geq \ell$. The collection of subfunctors
\[
\{\QRegA{\mathcal I,g}{\ell,s} \mid g \in \PGL,\mathcal I\subset\{1,\dots, N_m(s)\}\; \text{s.t} \; \langle \UC \rangle \in \mathbb{QS}_{p(z)}\}
\]
covers the Quot functor $\QuotFunctorReg{\ell}$.
\end{proposition}

\begin{proof}
We consider $\QuotFunctorReg{\ell}$ embedded by $\mathcal H^{[s]}$ in $\GrassScheme{s}$. By Proposition \ref{prop:QuasiStableCoverGr} we can im\-me\-dia\-te\-ly deduce that the Quot functor is covered by
\[
\left\{\QRegA{\mathcal I,g}{\ell,s}\mid g \in \PGL, \subset\{1,\dots, N_m(s)\} \; \text{s.t} \; \langle \UC \rangle \in \mathbb{QS} \right\}\,.
\]
  We obtain the statement proving that   $\QRegA{\mathcal I,\mathrm{Id}}{\ell,s}(A) \neq \emptyset$ for $\mathcal I\subset \{1, \dots, N(s)\}$ such that $\langle\UC\rangle\in \mathbb{ Q S}$  if and only if $\langle \UC \rangle \in
\mathbb{ Q S}_{p(z)} $. In fact  this  implies that for every $g\in  \PGL$ we have  $\QRegA{\mathcal I,g}{\ell,s}(A)\neq \emptyset$ if and only if $\langle \UC \rangle \in\mathbb{ Q S}_{p(z)}$.
 As this is a local and set-theoretical fact, we can assume that  $A$ is  a field.

Consider now a module $F \in \QRegA{\mathcal I,\mathrm{Id}}{\ell,s}(A)$, for $\mathcal I\subset \{1,\dots, N_m(s)\}$ such that $\langle \UC \rangle
\in \mathbb{QS}$.
Due to Lemma \ref{lem:grassFunctorUmarkedset} we know that $\langle F\rangle_s$ is generated by
a $\UC$-marked set. By Theorem \ref{th:markedSetChar} 
 we know that $ \langle F \rangle_{s'}$ contains an $A$-vector space of the same dimension
  as  $\langle \UC \rangle_{s'}$ for
every $s' \geq s$. This implies by Theorem \ref{th:markedSetChar} 
 that $N_m ({s'}) - p({s'}) = \dim(\langle F
\rangle_{s'}) \geq \dim(\langle \UC \rangle_{s'})$. But the growth theorem of
Macaulay \cite[Lem. 23]{Hulett1995}   implies that $\dim(\langle \UC
\rangle_{s'}) \geq N_m ({s'}) - p({s'})$, hence we have equality and the Hilbert
polynomial of $\langle \UC \rangle$ must be $p(z)$. 
\end{proof}

 In order to have an  open cover for $\QuotFunctorReg{l}$ made up of less open subsets than the one given in Proposition \ref{prop:coverQuot}, we need some preliminar results.

\begin{proposition}\label{cor:opCovQuotMarkedBasis} Consider $\mathcal I\subset\{1,\dots, N_m(s)\}$, such that $\langle \UC \rangle \in \mathbb{QS}_{p(z)}$ and $\reg(\langle \UC\rangle^{\sat})\leq s$.  Let $F$ be an element of  $\GrassFunctorL{\mathcal I,\mathrm{Id}}{s}(A)$. 
$F\in \QRegA{\mathcal I,\mathrm{Id}}{r,s}(A)$ if and only if   for every $s^\prime\geq s$ the
A-module    $\langle F\rangle_{s^\prime}$   is free of rank  $q(s^\prime)$  and
it is  generated by   $\langle \UC\rangle \cap \T_{s^\prime}^m$-marked
basis.
\end{proposition}

\begin{proof} Observe that in the present hypotheses, $\langle \UC\rangle$ is stable.
As the question is again local, we again assume that $A$ is a local ring. We first consider the special case with $A$ a field. 
Let $G=\{f_\tau\}_{\tau\in \UC}$ be the $\UC$-marked set generating $F$, and for every $s^\prime\geq s$ we denote by $G^{(s^\prime)}$ the set $\{x^\delta f_\tau\vert f_\tau \in G, \deg(x^\delta f_\tau)=s^\prime, \min(\tau)\geq \max(x^\delta)\}$.
It is immediate that $\langle G^{(s)} \rangle \subseteq F_s$. Using the same argument in the proof of Theorem \ref{prop:coverQuot}, the dimension of both vector
spaces is $q(s^\prime)$. Hence they must be equal for every degree $s^\prime$
and this implies via Theorem \ref{th:markedSetChar} 
 that $G$ is a $\UC$-marked
basis of $F$.

We generalize this result to the case $(A,\mathfrak m,  \Bbbk) $ a local ring  by Nakayama lemma, since for every $s'\geq s$ the $A$-module   $F_{s^\prime}$ contains the  free submodule $\langle G^{(s^\prime)}\rangle$ of rank $N_m(s')-p(s')$ (by Theorem \ref{th:markedSetChar}) 
  and the two $A/\mathfrak m$-vector spaces  $F_{s^\prime}\otimes_A  A/\mathfrak m  $ and $ \langle G ^{( s^\prime)}\rangle \otimes_A  A/\mathfrak m$  coincide as they have the same dimensions.
\end{proof}

We now easily prove that the functor $\QRegA{\mathcal I,\mathrm{Id}}{\ell,\ell}$ with
$\langle \UC \rangle \in \mathbb{QS}_{p(z)}^\ell$ is isomorphic to the functor
$\MFFunctor{\UC}^m$.
\begin{proposition}\label{cor:repQuotSubFunc}  Let $\mathcal I\subset\{1,\dots, N_m(\ell)\}$ be such that $\langle \UC\rangle \in\mathbb{QS}_{p(z)}^\ell$. 
\begin{enumerate}[(i)]
\item  \label{cor:repQuotSubFunc_i}The subfunctor $\QRegA{\mathcal I,\mathrm{Id}}{\ell,\ell}$ is isomorphic to the marked functor $\MFFunctor{\UC}^m$
\item \label{cor:repQuotSubFunc_ii} The subfunctor $\QRegA{\mathcal I,\mathrm{Id}}{\ell,\ell}$ is the functor of points of an an affine subscheme of the  affine space $\mathbb A^{p(\ell)\cdot q(\ell)}$.
\end{enumerate}
\end{proposition}
\begin{proof}
First, observe that $\langle \UC\rangle$ in the present hypotheses is stable, by Proposition \ref{prop:varieS}. 
Item \eqref{cor:repQuotSubFunc_i}  is a straighforward consequence of  Proposition \ref{cor:opCovQuotMarkedBasis}; item \eqref{cor:repQuotSubFunc_ii}  follows from \eqref{cor:repQuotSubFunc_i}  and Theorem \ref{thm:rappr}. 
\end{proof}

\begin{proposition}\label{cor:changedegree}
Let $U$ be a saturated quasi-stable module with Hilbert polynomial $p(z)$ and $\reg(U)\leq \ell$. We denote by $\mathcal I^{[s]}$ the set of indexes defining the module $U\cap \mathbb T_s^m$. Let $F=F^{\sat}$ be a module in $\QuotFunctorReg{\ell}$. For every $s,s'\geq \ell$,  $\langle F_s\rangle$ belongs to $\QRegA{\mathcal I^{[s]},\mathrm{Id}}{\ell,s}$ if and only if $\langle F_{s'}\rangle$ belongs to $\QRegA{\mathcal I^{[s']},\mathrm{Id}}{\ell,s'}$.
\end{proposition}
\begin{proof}
By Proposition \ref{cor:repQuotSubFunc} \eqref{cor:repQuotSubFunc_i}, the statement we intend to prove is equivalent to $\MFFunctor{\mathcal U^c_{\mathcal I^{[s]}}}^m\simeq \MFFunctor{\mathcal U^c_{\mathcal I^{[s']}}}^m$. For this isomorphism, we can repeat the arguments given in the proof of \cite[Theorem 3.4 (i)]{LR2}: indeed, all the arguments given in \cite{LR2} apply also in the stable case, and both $U\cap \mathbb T_s^m$ and $U\cap \mathbb T_{s'}^m$ are stable, by \ref{potenze} \eqref{importante_00}.
\end{proof}

We now prove that   the modules in $\mathbb{QS}^{\ell}_{p(z)}$ are sufficient to cover $\QuotFunctorReg{\ell}$, refining the result given in in Proposition \ref{prop:coverQuot}.

\begin{theorem}\label{thm:quasiStableCoverQuot}\ 
Consider $\ell\leq s$.
\begin{enumerate}[(i)]
\item \label{item:qsCoverQuot_i} Let $U=\langle U_s\rangle$ be a quasi-stable module in $\mathbb{QS}_{p(z)}^{\ell}$ and let $\mathcal I^{[s]}$ be as in Proposition \ref{cor:changedegree}. Then
$\QRegA{\mathcal I^{[s]},g}{\ell,s}= \QRegA{\mathcal I^{[r]},g}{\ell,r}=\QRegA{\mathcal I^{[r]},g}{r,r}$ as subfunctors of $\QuotFunctorReg{\ell}$.
\item The collection of subfunctors
\begin{equation}\label{eq:QFRegStableCover}
\left\{\QRegA{\mathcal I,g}{\ell,s} \mid g \in \PGL, \mathcal I\subset\{1,\dots, N_m(s)\} \; \text{s.t} \; \langle \UC \rangle \in \mathbb{QS}_{p(z)}^{\ell} \right\}
\end{equation}
covers the Quot functor with bounded regularity.
\end{enumerate}
\end{theorem}
\begin{proof}
\begin{enumerate}[(i)]
\item The equality between $\QRegA{\mathcal I^{[r]},g}{\ell,r}$ and $\QRegA{\mathcal I^{[r]},g}{r,r}$ is due to Corollary \ref{cor:RegBoundInMS}. We obtain the other equality by Proposition \ref{cor:changedegree}.

\item  By item \eqref{item:qsCoverQuot_i}, we can take $s=\ell$. 
As the question is local it is sufficient to consider the case that  the ring $A$
is  a field. Then Lemma \ref{lem:existsQS} applies.

\end{enumerate}
\end{proof}

\begin{corollary}
The Quot functor with bounded regularity  $\QuotFunctorReg{\ell}$ is an open subfunctor of $\QuotFunctor$.
\end{corollary}

\begin{definition}\label{def:quasiStableCoverQuot}
The \emph{quasi-stable cover} of $\QuotFunctorReg{\ell}$ is the collection of the open
subfunctors \eqref{eq:QFRegStableCover} of Theorem \ref{thm:quasiStableCoverQuot}.
\end{definition}

\begin{remark}
We constructed a cover of $\QuotFunctorReg{\ell}$ by using the quasi-stable modules in $\mathbb{QS}^{\ell}_{p(z)}$ and suitable deterministic changes of coordinates. There exists also a change of
coordinates to reach a Borel-fixed position, depending on $p=\mathrm{char}( \Bbbk)$ (for short, $p$-Borel fixed position). 
Therefore, we could repeat the statements and proofs of the present section in order to prove (constructively) the existence of a $p$-Borel cover of the
Quot functor, which is in general a more sparse cover 
than the quasi-stable one of Definition \ref{def:quasiStableCoverQuot}. However we prefer to consider the quasi-stable cover because this cover is independent of the characteristic and the
algorithm to reach the stable position is cheaper.

Furthermore in the next section we show that it is possible to compute equations for the open subscheme
of the Quot scheme corresponding to each quasi-stable open subfunctor. The
computational cost to get such  equations for   open neighborhoods of a given
point of the Quot scheme can be significantly different depending on the
neighborhood we choose. Hence it is an advantage to have a relatively dense
cover in order to choose the more convenient neighbourhood.
\end{remark}

\section{Representability of Quot functors and equations}\label{sec:GlobalEquations}

From now on, we will consider only  open subfunctors $\QRegA{\mathcal I, \mathrm{Id}}{\ell,s}$, with $\mathcal I\subset \{1,\dots, N_m(s)\}$
such that $\langle \UC \rangle\in \mathbb{QS}_{p(z)}^{\ell}$.

It was proved by Grothendieck that the Grassmann and Quot functors are
representable. We now  prove that the  {Quot functor with bounded regularity} is represented by a locally closed subscheme of the Grassman scheme, using the quasi-stable open cover  and the fact that the Quot functor is a Zariski sheaf.

\begin{theorem}\label{thm:rappQFReg}
The Quot functor with bounded regularity  is the functor of points of a  closed
subscheme $\mathrm{Quot}_{p(z)}^{m,[\ell]}$ of  $\GrassScheme{s}\setminus{L}^{[\ell,s]}_{p(z)}$.
\end{theorem}
\begin{proof}

By the semicontinuity theorem for regularity, for every  $\ell'<\ell$,  $\QuotFunctorReg{\ell' }$ can be considered as an open subfunctor $\QuotFunctorReg{\ell}$. Furthermore,   $\QuotFunctor$ is
a Zariski sheaf \cite[Section 5.1.3]{nitsure}. Hence,  $\QuotFunctorReg{\ell}$ is a Zariski sheaf.

By  \cite[Theorem VI-14]{EH} it
suffices to check the representability on an open cover of  $\QuotFunctorReg{\ell}$: we choose the one given in
\eqref{eq:QFRegStableCover}. By Theorem \ref{thm:quasiStableCoverQuot} \eqref{item:qsCoverQuot_i}
and by Proposition \ref{cor:repQuotSubFunc} \eqref{cor:repQuotSubFunc_ii}, we immediately conclude  that the Quot functor with bounded regularity is the functor of points of a scheme. By Proposition \ref{prop:coverGrMinusL} and by Theorem \ref{th:markedSetChar} \eqref{it:markedSetChar_ix}, we obtain that the scheme representing the Quot functor with bounded regularity is a closed  subscheme of  $\GrassScheme{s}\setminus{L}^{[\ell,s]}_{p(z)}$.
\end{proof}

\begin{corollary}\label{cor:QuotRepr}
The Quot functor $\QuotFunctor$ is the functor of points of a scheme and it is also a closed subscheme of $\GrassScheme{s}$.
\end{corollary}
\begin{proof}
It is sufficient to observe that $\QuotFunctor=\QuotFunctorReg{r}$ and use Theorem \ref{thm:rappQFReg}, observing that $L^{[r,s] }_{p(z)}=\emptyset$.
\end{proof}

We will now  give a constructive procedure to define an ideal $\mathfrak H\subset  \Bbbk[\Delta]$ such that $\Proj( \Bbbk[\Delta]/\mathfrak H)=\QuotScheme$, where $\Bbbk[\Delta]$ is the ring of Plucker coordinates for $\GrassScheme{r}$. The construction of $\mathfrak H$ starts from the ideals defining the affine schemes representing the open subsets $\QRegA{\mathcal I, \mathrm{Id}}{r,r}$ of \eqref{eq:QFRegStableCover}. Since $\QuotFunctor=\QuotFunctorReg{r}$, the cover of Definition \ref{def:quasiStableCoverQuot} is in this case indexed by $\mathbb{QS}_{p(z)}=\mathbb{QS}_{p(z)}^{r}$. Furthermore, if $U\in \mathbb{QS}_{p(z)}$, then $U_s$ is stable, being $s\geq r$.

For every $\mathcal I\subset\{1,\dots, N_m(s)\}$ such that $\langle \UC\rangle \in \mathbb{QS}_{p(z)}$, we denote by $ \Bbbk[C_{\mathcal I}]/(\mathcal
{R}_{\mathcal I})$ the quotient ring that defines the affine scheme representing $\QRegA{\mathcal I, \mathrm{Id}}{r,r}$ (see Theorem \ref{thm:rappr} and Proposition \ref{cor:repQuotSubFunc}).

As shown in Theorem \ref{thm:rappr} and Proposition \ref{cor:repQuotSubFunc}, $\Spec( \Bbbk[C_{\mathcal I}]/(\mathcal
{R}_{\mathcal I}))$ is the open subset of $\QuotScheme$ corresponding to the locus where $\Delta_{\mathcal I}$ can be inverted. Hence, the ideal $(\mathcal R_{\mathcal I})$ is the dehomogenization of the ideal in $ \Bbbk[\Delta]$ that defines $\QRegA{\mathcal I,\mathrm{Id}}{r,r}$ as a closed subscheme of $\GrassScheme{r}\setminus\{\Delta_{\mathcal I}\neq 0\}$.

We construct an ideal $\mathfrak h_{\mathcal I}\subset\Bbbk[\Delta]$ starting from $\mathcal  R_{\mathcal I}$: \\
Let  $\mathfrak G_{\mathcal I}\subseteq \Bbbk[C_{\mathcal I}][\mathbf x]$ be the $\UC$-marked set as defined in \eqref{polymarkKC}. We consider the $p(r)\times N_m(r)$-matrix $\mathcal M$, with columns indexed by the terms in $\mathbb T^m_r$ and rows indexed by the set $\mathcal  U_{\mathcal I}$. The columns of $\mathcal M$ contain the coefficients of the polynomials $h_{ij}$ such that $x^{\alpha_i}e_j \xrightarrow{\mathfrak G_{\mathcal I}^{(r)}}_* h_{ij}$, for every $x^{\alpha_i}e_j\in \mathbb T^m_r$.\\
Consider now the set of equations 
\begin{equation}\label{eq:CtoPl}
\{\Delta_{\mathcal J}-\mathcal M_{\mathcal J}\vert\mathcal J\subset\{1,\dots, N_m(r)\}, \vert \mathcal J\vert=p(r)\}
\end{equation}
 where $\mathcal M_{\mathcal J}$ is the minor corresponding to the columns with indexes $j\in \mathcal J$.
Consider   the set $\mathcal R_{\mathcal I}\subset\Bbbk[C_{\mathcal I}]$ 
 and compute the complete Gr{\"o}bner reduction of the set $\mathcal R_{\mathcal I}$ with respect to the set of polynomials  in \eqref{eq:CtoPl} using an elimination term order for the variables $C_{\mathcal I}$. We obtain a set of non-homogeneous polynomials in $ \Bbbk[\Delta]$. We homogenize each polynomial in this set with $\Delta_{\mathcal I}$ and take these homogeneous polynomials as  generators for an ideal in $\Bbbk[\Delta]$ that we denote by $\mathfrak h_{\mathcal I}$.

We define $\mathfrak h:=\bigcup_{\mathcal I\vert \UC\in \mathbb{QS}_{p(z)}} \mathfrak h_{\mathcal I}$. 
Moreover, we consider for every $g\in \PGL$ the set of equations $\mathfrak h^{g}$ obtained
by the action of $g$ on the elements of $\mathfrak h$. Finally we define the ideal
\[\mathfrak H:= \mathfrak P\cup\left(\bigcup_{g\in \PGL}\mathfrak h^{{g}}\right),\]
where $\mathfrak P$ is the ideal $\Bbbk[\Delta]$ containing the Plucker relations, that is $\Proj(\Bbbk[\Delta]/\mathfrak P)=\GrassScheme{r}$.

\begin{theorem}\label{thm:QuotClosedInGrass}
Let $p(z)$ be an admissible Hilbert polynomial for submodules of $\Sk^m$. The homogeneous ideal $\mathfrak H$ in the ring of Pl{\"u}cker coordinates $ \Bbbk[\Delta]$ of the Pl{\"u}cker embedding $\GrassScheme{r}\hookrightarrow \mathbb P^{N_m(r)\choose p(r)}$
 defines $\QuotScheme$ as a closed subscheme of $\GrassScheme{r}$.
 \end{theorem}
\begin{proof}
We follow the lines of the proof of \cite[Theorem 6.5]{BLMR} on the Hilbert scheme.

For convenience, we denote by $\mathcal Z$ the subscheme of $\GrassScheme{r}$
defined by $\mathfrak H$ and by $\mathfrak D$ the saturated ideal in $ \Bbbk[\Delta]$ that defines $\QuotScheme$. We will show that $Z=\QuotScheme$, although in general $\mathfrak H\neq \mathfrak D$. 

As equality of subschemes is a local property, we can check the equality locally. 
The proof is divided in two steps.
\begin{description}
\item[Step 1.] For every $\mathcal I\subset\{1,\dots,N_m(r)\}$ such that $\langle\UC\rangle\in \mathbb{QS}_{p(z)}$, the ideal generated by $\mathfrak h_{\mathcal I}$ defines the affine scheme representing $\QRegA{\mathcal I,\mathrm{Id}}{r,r}$
as closed subscheme of the scheme $\mathrm{G}_{\mathcal I,\mathrm{Id}}^{[r]}$ of Corollary \ref{cor:SchCoverGr}, representing $\GrassFunctorL{\mathcal I,\mathrm{Id}}{r}$.
\item[Step 2.] For every (closed) point $F$ of $\GrassScheme{r}$, $\mathcal Z$ and $\QuotScheme$ coincide on a neighborhood of $\langle F\rangle$.
\end{description}

{\bf Proof of Step 1.} We have to prove that for every $\mathcal I\subset\{1,\dots,N_m(r)\}$ such that $\langle\UC\rangle\in \mathbb{QS}_{p(z)}$, and for every $ \Bbbk$-algebra $A$ and $F$  belonging to $\GrassFunctor{r}(A)$, $\langle F\rangle\subset \Sk^m$ belongs to $\QRegA{\mathcal I,\mathrm{Id}}{r,r}$ if, and only if, the polynomials in $\mathfrak H$ vanish when evaluated at $\langle F\rangle$.
Referring to Proposition \ref{cor:opCovQuotMarkedBasis}, Theorem \ref{thm:rappr} and Proposition \ref{cor:repQuotSubFunc}, it is sufficient to observe that the vanishing at $\langle F\rangle$ of the polynomials in $\mathfrak H$ is equivalent to the vanishing at $\langle F\rangle$ of the polynomials in $(\mathcal R_{\mathcal I})$.

{\bf Proof of Step 2.} Both ideals $\mathfrak H$ and $\mathfrak D$ are invariant under the action of $\PGL$, $\mathfrak H$ by construction and $\mathfrak D$ because  $\QuotScheme$ is. Being $ \Bbbk[\Delta]$ a Noetherian ring, we can choose $h_1,\dots, h_d$ generators of the ideal $\mathfrak H$. More precisely, we denote by $g_i$ the element in $\PGL$ such that $h_i\in \mathfrak h^{{g_i}}$. Hence, $\mathfrak h^{{g_1}}\cup\cdots \cup \mathfrak h^{{g_d}}=\mathfrak H$. Being $\mathfrak H$ invariant under the action of $\PGL$, we get for every $g\in \PGL$
\[
\mathfrak h^{{g g_1}}\cup\cdots \cup \mathfrak h^{{g g_d}}=\left(\mathfrak h^{{g_1}}\cup\cdots \cup \mathfrak h^{ {g_d}}\right)^{{g}}=\mathfrak H^{{g}}=\mathfrak H.
\]
Using the invariance of $\mathfrak D$ under the action of $\PGL$ and by Step 1., if we restrict to the open subset $\mathrm{G}_{\mathcal I,gg_1}^{[r]}\cap \cdots \cap \mathrm{G}_{\mathcal I,gg_1}^{[r]}$, the ideals $\mathfrak H$ and $\mathfrak D$ define the same scheme, hence
\[
\QuotScheme\cap \left(\mathrm{G}_{\mathcal I,gg_1}^{[r]}\cap \cdots \cap \mathrm{G}_{\mathcal I,gg_1}^{[r]}\right)=\mathcal Z\cap \left(\mathrm{G}_{\mathcal I,gg_1}^{[r]}\cap \cdots \cap \mathrm{G}_{\mathcal I,gg_1}^{[r]}\right).
\]
It only remains to prove that for every $F\in \GrassScheme{r}$, there is $\mathcal I\subset\{1,\dots,N_m(r)\}$ such that $\langle \UC\rangle\in \mathbb{QS}$ and there is $g\in \PGL$ such that $\langle F\rangle \in \mathrm{G}_{\mathcal I,gg_1}^{[r]}\cap \cdots \cap \mathrm{G}_{\mathcal I,gg_1}^{[r]}$.

By Proposition \ref{prop:QuasiStableCoverGr} there is $\overline{\mathcal I}\subset\{1,\dots,N_m(r)\}$ such that $\langle \mathcal {U}^c_{\overline{\mathcal I}}\rangle\in \mathbb{QS}$ and there is $\overline g\in \PGL$ such that $\langle F\rangle \in \mathrm{G}_{\overline{\mathcal I},\overline{g}}^{[r]}$. Since $\mathrm{G}_{\overline{\mathcal I},\overline{g}}^{[r]}$ is an open subset of $\GrassScheme{r}$, an open subset of the orbit of $\langle F\rangle$ under the action of $\PGL$ is contained in $\mathrm{G}_{\overline{\mathcal I},\overline{g}}^{[r]}$: there is an open subset $\mathcal G$ of $\PGL$ such that for every 
$g' \in \mathcal G$, $g'^{-1}\centerdot F\in \mathrm{G}_{\overline{\mathcal I},\overline{g}}^{[r]}$, in other words $F\in \mathrm{G}_{\overline{\mathcal I},g'\overline{g}}^{[r]}$. Hence, for a general $g\in \PGL$,  $gg_1\overline g,\dots, g g_d \overline g \in \mathcal G$ and $F\in \mathrm{G}_{\mathcal I,gg_1}^{[r]}\cap \cdots \cap \mathrm{G}_{\mathcal I,gg_1}^{[r]}$ as wanted.

\end{proof}

\begin{remark}
We can rephrase the construction of the ideal $\mathfrak H$ and the statement of Theorem \ref{thm:QuotClosedInGrass} in order to obtain the ideal defining the scheme representing the functor $\QuotFunctorReg{\ell}$ as a closed subscheme of $\GrassScheme{s}\setminus{L}^{[\ell,s]}_{p(z)}$: it is sufficient to use in the construction the set of quasi-stable modules in $\mathbb{QS}_{p(z)}^{\ell}$.
\end{remark}

\section{An example: Computations on $\mathrm{Quot}^{2}_{2}$ on $\mathbb P^1$}\label{sec:example}

We consider $p(z)=2$, $n=1$ and $m=2$. This is the very first example one can think about in order to consider a non-trivial Quot scheme which is not a Hilbert scheme. Nevertheless, even if this is the simplest case on which we can test our methods, up to our knowledge nothing is known about this Quot scheme. In this section, we give a description of the construction of the ideals $\mathcal R_{\mathcal I}$ defining the  open cover of of Definition \ref{def:quasiStableCoverQuot} and the global equations defining $\mathrm{Quot}^{2}_{2}$. A detailed description on the geometry of $\mathrm{Quot}^{2}_{2}$ can be found in \cite{BRS}.

We consider  the scheme $\mathrm{Quot}^{2}_{2}$, that parameterizes the saturated submodules of $A[x_0,x_1]^2$ with constant Hilbert polynomial 2. The Gotzmann number is $r=2$, hence we will study this Quot scheme under the embedding in $\mathrm{Gr}^6_2$. We will also study $\mathrm{Quot}_2^{2,[1]}$, which is the  subscheme of $\mathrm{Gr}^6_2\setminus L_2^{1,2}$ whose functor of points is $\mathcal Q\mathrm{uot}_2^{2,[1]}$.

Consider
\[
U_1=(x_1^2) e_1\oplus(1) e_2, \quad U_2=(x_1)e_1\oplus (x_1)e_2, \quad U_3= (1)e_1\oplus(x_1^2) e_2.
\]

We have $\mathbb{QS}_2=\{U_1,U_2,U_3\}$ and $\mathbb{QS}_2^1=\{U_2\}$.

\subsection{Global equations for $\mathrm{Quot}^{2}_{2}$ in $\mathbb P^{14}$ }
 We keep on using the notation $U_i$ for the embedding of the quasi-stable module $U_i$ in $\mathrm{Gr}^6_2$, for $i=1,2,3$.

By the procedure of Section \ref{subsec:MarkedScheme}, and in particular by Theorem \ref{thm:rappr}, we explicitely construct the affine scheme representing $\MFScheme{\mathcal P(U_i)}$, $i\in\{1,2,3\}$. 

We construct the marked scheme on $\mathcal P(U_1)$ starting from the following marked set:
\[f_1=({x}^{2}-C_{{1,1}}xy-C_{{1,2}}{y}^{2})e_1,\quad
f_2=-(C_{{2,1}}xy+C_{{2,
2}}{y}^{2})e_1+{x}^{2}e_2,\]
\[
f_3=-(C_{{3,1}}xy+C_{{3,2}}{y}^{2})e_1+xye_2,\quad
f_4=(-C_{{4,1}}xy-C
_{{4,2}}{y}^{2})e_1+{y}^{2}e_2.
\]
We compute the ideal $\mathcal R_1$, obtaining
\[
\mathcal R_1=(-C_{{1,2}}C_{{3,1}}+C_{{2,2}},C_{{1,1}}C_{{4,1}}-C_{{3,1}}+C_{{4,2}},
-C_{{1,1}}C_{{3,1}}-C_{{1,2}}C_{{4,1}}+C_{{2,1}},-C_{{1,2}}C_{{4,1}}+C
_{{3,2}}
)
\]
In the same way, we can construct the ideals $\mathcal R_2$ and $\mathcal R_3$ which define the schemes representing $\MFScheme{\mathcal P(U_2)}$ and $\MFScheme{\mathcal P(U_3)}$.

Each of these 3 ideals has 4 generators and each of them allows the elimination of a variable (in the sense of Groebner theory), hence for every $i\in\{1,2,3\}$, $\MFScheme{\mathcal P(U_i)}\simeq \mathbb A^4$.

Following the construction outlined in Section \ref{sec:GlobalEquations}, we can compute the ideal $\mathfrak h$   in the polynomial ring $ \Bbbk[\Delta]$, where $\Delta$ is the set of Plucker coordinates of $\mathrm{Gr}^6_{2}$, $\vert\Delta\vert=15$. We then repeatedly apply  some random elements  $g_i\in\PGL$ on the ideal $\mathfrak h$, until for some $t$
\[\mathfrak h^{g_{1}}\cup\cdots \cup \mathfrak h^{g_{t}}\cup \mathfrak h^{g_{t+1}}=\mathfrak h^{g_{1}}\cup\cdots \cup \mathfrak h^{g_{t}}.\] 
By noetherianity, such a $t$ exists, and on this specific example $t=4$.

Adding the Plucker relations, we obtain the ideal defining $\mathrm{Quot}^2_{2}$ as a closed subscheme of $\mathbb P^{14}$. We can exhibit a set of generators made up of 61 polynomials of degree 2, 3 and 4. This Quot scheme has Hilbert polynomial
\begin{equation}\label{eq:hpolyQuot22}
\frac{11}{12}z^4+\frac{11}{3} z^3+\frac {67}{12} z^2+\frac {23}{6}z+1,
\end{equation}
hence it is a fourfold in $\mathbb P^{14}$ of degree 22.

\subsection{Global equations for $\mathrm{Quot}_{2}^{2,[1]}$ in $\mathbb P^{14}$}

By Theorem \ref{thm:quasiStableCoverQuot} \eqref{item:qsCoverQuot_i}, $\mathrm{Quot}_{2}^{2,[1]}$ embeds in $\mathrm{Gr}_2^4$, which embeds in $\mathbb P^5$. In this case,  $\mathrm{Quot}_{2}^{2,[1]}$ is simply the open subset $\mathbb A^4\simeq  \mathrm{Gr}_2^4\setminus L_2^{1,1}$. Indeed, in this case $L_2^{1,1}$ is defined by the ideal $(\Delta_{3,4})$.

We can also compute the equations defining $\mathrm{Quot}_{2}^{2,[1]}$ as an open subscheme of $\mathrm{Gr}_2^6\setminus L_2^{1,2}$. It is sufficient to consider only one marked scheme, the one defined by the ideal $\mathcal R_2$ and use the procedure described in Section \ref{sec:GlobalEquations}. After homogenizing the generators of $\mathcal R_2$ in $\Bbbk [\Delta]$, we obtain the ideal $\mathfrak h'$. We apply 4 times random elements $g_i\in \PGL$ on $\mathfrak h'$, obtaining the ideal $\mathfrak H'$ defining  $\mathrm{Quot}_{2}^{2,[1]}$ as a subscheme of $\mathrm{Gr}_2^6\setminus L_2^{1,2}$ in $\mathbb P^{14}$.  

In this case, $L_2^{1,2}=(\Delta_{3,6},\Delta_{1,4})$ and the closed scheme defined by $\mathfrak H'$, which contains the  scheme $\mathrm{Quot}^{2,[1]}_2$,  has Hilbert polynomial
\begin{equation}\label{eq:HPolyLocus}
{\frac {11}{12}}z^4+5z^3+{\frac {67}{12}}z^2+\frac{7}{2}z+1.
\end{equation}
The construction of $\mathfrak H'$ is quite faster than that of $\mathfrak H$, due to the fact that we have only one open subset in the open cover of $\mathrm{Quot}_2^{2,[1]}$, up to the action of $\PGL$.
Nevertheless
the ideal $\mathfrak H'$ by construction is contained in the ideal $\mathfrak H$ that defines $\mathrm{Quot}^2_2$, and the Hilbert polynomial \eqref{eq:HPolyLocus} of $\Proj(\Bbbk[\Delta]/\mathfrak H')$ is smaller than the one computed for $\mathrm{Quot}_2^2$ in \eqref{eq:hpolyQuot22}, hence $\Proj(\Bbbk[\Delta]/\mathfrak H')\supset \mathrm{Quot}^2_2$. 
Indeed, $\mathfrak H'$ defines a closed scheme that strictly contains $\mathrm{Quot}^{2,[1]}_2$. We have that  $\Proj(\Bbbk[\Delta]/\mathfrak H')\setminus\mathrm{Quot}^{2,[1]}_2\subset L_2^{1,2}$.

\section{Conclusions}

In this paper, we defined and investigated properties of marked bases over
a quasi-stable monomial module $U\subseteq\Sk^m_{\mathbf d}$. The family
of all modules generated by a marked basis over $\mathcal P(U)$ possesses a
natural structure as an affine scheme (Theorem \ref{thm:rappr}). In particular, we proved that the quasi-stable module $U$ provides upper
bounds on some homological invariants of any module generated by a
$\mathcal P(U)$-marked basis such as Betti numbers, regularity or
projective dimension (Corollary \ref{cor:bounds}).

We exploited these properties and constructions
to obtain local and global equations of Quot schemes and of special loci
of them, those given by an upper  bound on the Castelnuovo-Mumford regularity of a module.  Indeed, we proved that we have an open cover of a Quot scheme (resp. of its locus defined by an upper bound on the regularity) whose open subsets are suitable marked schemes over a quasi-stable module (Theorem \ref{thm:quasiStableCoverQuot}). Starting from this open cover, we obtained global equations defining a Quot scheme (resp. its locus defined by an upper bound on the regularity) as a closed (resp. locally closed) subscheme of a suitable projective space (Theorem \ref{thm:QuotClosedInGrass}).

  In future, inspired by Corollary \ref{cor:bounds}, we intend to investigate other loci of a Quot scheme, given by bounds on other numerical invariants of a module, such as projective dimension, extremal Betti numbers. In order to obtain similar results to those for the locus with bounded regularity, we will need also other tools, since  a preliminary study showed, for instance, that the locus given by a bound on projective dimension in general is not an open subset of a Quot scheme.

We will also investigate some explicit examples of Quot schemes, as we are doing in \cite{BRS}, in order to have a better comprehension of the geometry of a Quot scheme, using explicit equations defining it, locally or globally.

\section*{Acknowledgements}

The first author was supported by a fellowship by the Otto-Braun-Fonds.
The second and third author are member of GNSAGA (INdAM, Italy), which
partially funded the second author during the preparation of this
article. The third author was partially supported by the framework of PRIN
2010-11 \lq\lq Geometria delle variet\`a\ algebriche\rq\rq, cofinanced by
MIUR. The work of the fourth author was partially performed as part of the
H2020-FETOPEN-2016-2017-CSA project $SC^{2}$ (712689). Mutual visits of the
authors were funded by the Deutsche Forschungsgemeinschaft.

\bibliographystyle{plain}

\end{document}